\documentclass[final]{siamltex}

\usepackage{amsmath,amssymb}
\usepackage{color,graphicx,mathrsfs,tikz,amssymb}
\usetikzlibrary{calc}
\usepackage[bbgreekl]{mathbbol}
\usepackage{mathtools}   
\usepackage[all]{xy}
\usepackage{tikz}
\usetikzlibrary{cd}

\def\d{\partial}
\newcommand{\F}[1]{\mathbb{\Lambda}^{{#1}}}
\newcommand{\Hd}[1]{H(d, \om, \F {#1})}

\newcommand{\RRR}{\mathbb{R}}
\newcommand{\KKK}{\mathbb{K}}
\newcommand{\MMM}{\mathbb{M}}
\newcommand{\skw}{\mathop{\text{skw}}}
\newcommand{\tr}[1]{{\mathrm{tr}^{(#1)}}}
\newcommand{\om}{\Omega}
\newcommand{\oh}{\Omega_h}
\newcommand{\veps}{\varepsilon}

\newcommand{\ds}{\displaystyle}

\newcommand{\p}[2]{[{#2}]^{(#1)}}
\newcommand{\Grad}{\mathop{\mathrm{Grad}}}
\renewcommand{\grad}{\mathop{{\mathrm{grad}}}}
\newcommand{\curl}{\mathop{\mathrm{{curl}}}}
\renewcommand{\div}{\mathop{\mathrm{{div}}}}
\newcommand{\Div}{\mathop{\mathrm{{Div}}}}
\newcommand{\Curl}{\mathop{\mathrm{Curl}}}

\newcommand{\vpi}{\varPi}

\newcommand{\DD}{\mathcal{D}}
\newcommand{\II}{\mathcal{I}}
\newcommand{\RR}{\mathcal{R}}
\renewcommand{\SS}{\mathcal{S}}
\newcommand{\BB}{\mathcal{B}}

\newcommand{\PP}{\mathcal{P}}

\newcommand{\vPi}{\varPi}
\newcommand{\Qhk}{Q_h^{(k)}}
\newcommand{\Qh}{Q_h}
\newcommand{\Dhk}{D_h^{(k)}}
\newcommand{\Dh}[1]{D_h^{({#1})}}
\newcommand{\Pk}{{P}^{(k)}}
\newcommand{\Pkm}{{P}^{(k-1)}}
\newcommand{\Bhk}{\mathcal{B}_h^{(k)}}
\newcommand{\Bh}[1]{\mathcal{B}_h^{(#1)}}
\newcommand{\Bhh}{\mathcal{B}_h}

\newcommand{\og}{\omega}

\newcommand{\Vnk}{V_h^{(0),k}}
\newcommand{\Bnk}{B^{(0),k}}
\newcommand{\Bn}[1]{B^{(0),{#1}}}
\newcommand{\Ank}{A^{(0),k}}
\newcommand{\An}[1]{A^{(0),{#1}}}
\newcommand{\Vhk}{\smash[t]{{V_h^{(k)}}}}
\newcommand{\Vh}[1]{{V_h^{({#1})}}}
\newcommand{\Vp}[1]{\mathcal{V}_h^{({#1})}}
\newcommand{\trg}[2]{\triangle({#1}, {#2})}

\newcommand{\Vt}{\tilde{V}}
\newcommand{\vt}{\tilde{v}}
\newcommand{\Rt}{\tilde{R}}
\newcommand{\vLt}{\tilde{\varLambda}}
\newcommand{\vL}{{\varLambda}}
\newcommand{\A}[1]{A^{({#1})}}
\newcommand{\B}[1]{B^{({#1})}}

\newcommand{\pihk}{\smash[t]{\vPi_h^{(k)}}}
\newcommand{\pih}[1]{\vPi_h^{({#1})}}
\newcommand{\pihh}{\vPi_h}
\newcommand{\diam}{\mathop{\text{diam}}}
\newcommand{\spn}{\mathop{\text{span}}}
\newcommand{\hK}{\hat K}
\newcommand{\dk}{d^{(k)}}
\newcommand{\dd}[1]{d^{({#1})}}
\newcommand{\dhh}[1]{d_h^{({#1})}}
\newcommand{\alphak}{\alpha_k}
\newcommand{\alphakm}{\alpha_{k-1}}

\newcommand{\piCurl}{\vPi_h^{\text{Curl}}}
\newcommand{\piDiv}{\vPi_h^{\text{Div}}}
\newcommand{\pidiv}{\vPi_h^{\text{div}}}

\newcommand{\rev}[1]{#1}

\title{The auxiliary space preconditioner  for the de Rham complex
\thanks{This work was performed under the auspices of the
  U.S. Department of Energy by Lawrence Livermore National Laboratory
  under Contract DE-AC52-07NA27344 and was also supported in part by
  the AFOSR under grant FA9550-17-1-0090 and by the ARO under US Army
  Federal Grant W911NF-15-1-0590. The initial numerical studies
  were facilitated by equipment acquired under NSF's Major Research
  Instrumentation grant DMS-1624776 and under the  
  Defense University Research Instrumentation Program provided  
  by the ARO Federal Grant W911NF-16-1-0307.}
}

\author{
J.~Gopalakrishnan\thanks{Portland State University, 
  PO Box 751 (MTH), Portland, OR 97207-0751,
  ({\tt gjay@pdx.edu})}
\and
M.~Neum\"uller\thanks{Institute of Computational Mathematics,
  Johannes Kepler University, 
         Altenbergerstr. 69, A-4040 Linz, Austria, 
         ({\tt neumueller@numa.uni-linz.ac.at})}
\and
P.~S. Vassilevski\footnotemark[2]
    \thanks{Center for Applied Scientific Computing, LLNL, Mail Stop L-561,
    Livermore, CA 94550, ({\tt panayot@llnl.gov})}
}

\begin{document}

\maketitle 

\begin{abstract}
  We generalize the construction and analysis of auxiliary space
  preconditioners to the $n$-dimensional finite element subcomplex of
  the de Rham complex.  These preconditioners are based on a
  generalization of a decomposition of Sobolev space functions into a
  regular part and a potential.  A discrete version is easily
  established using the tools of finite element exterior calculus.  We
  then discuss the four-dimensional de Rham complex in detail. By
  identifying forms in four dimensions (4D) with simple proxies, form
  operations are written out in terms of familiar algebraic operations
  on matrices, vectors, and scalars.  This provides the basis for our
  implementation of the preconditioners in 4D.  Extensive numerical
  experiments illustrate their performance, practical scalability, and
  parameter robustness, all in accordance with the theory.
\end{abstract}



\pagestyle{myheadings}
\thispagestyle{plain}
\markboth{}{}

\section{Introduction}\label{section: Introduction}

The auxiliary space preconditioners for problems posed in $H(\curl)$
and $H(\div)$, initially studied
by Hiptmair and Xu \cite{HiptmXu07}, are now
well understood both theoretically and practically, in two and three
space dimensions.  These preconditioners have been used for
accelerating a wide variety of solution techniques, thanks to their
highly scalable parallel implementations, known as AMS and ADS
preconditioners (see the software libraries HYPRE~\cite{hypre} and
MFEM~\cite{mfem}).  The goal of the present work is two-fold.  First,
we generalize the mathematical design and analysis of these
preconditioners to $n$ dimensions. Second, we provide an
implementation of the preconditioners in 4D and detail the techniques
we used to transform 4D exterior calculus into matrix and vector
operations.

An important ingredient in the analysis of the auxiliary space
preconditioners in two and three dimensions was the so-called regular
decomposition, which splits a Sobolev space function into a component
of higher regularity and a scalar or vector potential. Such
decompositions were known early on~\cite{BirmaSolom90}.  But the key
to the success of the auxiliary space preconditioners was a discrete
version of this decomposition found in~\cite{HiptmXu07}, now also
known as the HX decomposition.  Its practical use was elaborated in
\cite{KolevVassi09} and \cite{KolevVassi12}, where slightly stronger
results were established (using~\cite{zhao_pasciak}) to prove robustness
of the solvers in a general setting involving a stiffness term and a
mass term weighted with a parameter.  Further solvers in $H(\curl)$
and $H(\div)$ were developed in~\cite{bochev_et_al} and
\cite{yunrong}.

One of the motivations for this work, especially our 4D
implementation, is the recent increased interest in spacetime
discretizations. In three space dimensions, they yield large linear
systems built on 4D meshes and discretizations.  Starting as early as
the eighties, literature on spacetime methods began to accumulate
\cite{BabuskaJanik89, BabuskaJanik90, Hansbo94,
  JohnsonNaevertPitkaeranta84, JohnsonSaranen86, VegtVen02}.  As
methods that parallelize only spatial degrees of freedom created
increasingly larger computational bottlenecks in temporal
simulations~\cite{FalgoFriedKolev14}, the potential for higher
scalability of the spacetime methods received more attention,
resulting in a resurgence of interest in recent years \cite{
  Andreev13, BabuskaJanik90, BankVassilevskiZikatanov17,
  LangerMooreNeumueller16, LarssonMolten17, Mollet14,
  NeumuellerSteinbach11, NeumuellerVassilevskiVilla2017,
  SchwabStevenson09, Smears16, Steinbach15, UrbanPatera14}.
Further reasons for pursuing spacetime
discretizations, such as limited regularity~\cite{DemkoGopalNagar17}
and spacetime adaptivity~\cite{GopalSepul17} have also been noted.
Among these reasons, perhaps the most relevant to this work is the
above-mentioned potential of spacetime methods to break through
temporal causality barriers when exploiting parallelism. However, this
potential is unlikely to be realized without highly scalable solvers.
In turn, spacetime solvers in 4D are unlikely to be developed without
a complete understanding of preconditioners for the norm generated by
each of the four canonical first order partial differential operators
in 4D. Herein lies one of our contributions. By  showing how to build
scalable preconditioners for the norm of all the first order Sobolev
spaces in 4D, we  provide building blocks for designing spacetime
solvers.

\rev{To describe a specific scenario illustrating the need for
  preconditioners in 4D, recall that conservation laws take the form
  $\div F = 0$ for some flux $F$ depending on the unknown fields.
  Here, ``div'' is the 4D spacetime divergence when the conservation
  law in posed in three space dimensions. One can construct a
  spacetime discretization for this equation, following along the
  lines of~\cite{NeumuellerVassilevskiVilla2017} for scalar
  conservation laws. The resulting system of equations, as shown
  in~\cite{NeumuellerVassilevskiVilla2017}, is of saddle-point
  form. Its leading blocks on the diagonal correspond to bilinear
  forms that are equivalent to the canonical norms arising from the 4D
  de Rham sequence.  Therefore, a block-diagonal preconditioner for
  that saddle point system is obtained using diagonal blocks
  consisting of preconditioners for the relevant canonical 4D
  norms. This shows an immediate impact of our preconditioners in
  Section~\ref{section: The preconditioner} on existing work.  Our
  later discussions on 4D implementation are also of immediate
  relevance to this example.  Indeed, one of the solvers considered in
  \cite{NeumuellerVassilevskiVilla2017} utilizes iterations in a
  divergence-free space, which benefits from explicit knowledge of
  that subspace. Our considerations in Section~\ref{sec:impl}
  characterize this subspace as $\Div$ of certain skew-symmetric
  matrix-valued functions (where $\Div$ defined later --see
  \eqref{eq:Div}-- is such that $\div \circ \Div$ applied to
  skew-symmetric matrix-valued functions vanishes).  Beyond these
  comments, we shall not dwell on further details of applications in
  this paper.}

The remainder of the paper is structured as follows.  We begin in
Section~\ref{section: Preliminaries} with the necessary background on
finite element exterior calculus and introduce the regular
decomposition in $n$-dimensions. This section also reviews a few new
tools available thanks to the recent intensive research on finite
element exterior calculus, such as the bounded cochain projections and
their commutativity and approximation properties.
Section~\ref{section: The preconditioner}, introduces the auxiliary
space preconditioner, which is the main object of this study.  After
its definition and complete analysis, we proceed to
Section~\ref{sec:impl}, which specializes the discussion to 4D
exterior calculus and presents techniques and identities used for the
implementation of the preconditioner and its 4D ingredients.
Section~\ref{sec:numer} contains a large set of numerical results
illustrating the scalable and robust performance of the method, all in
accordance with the theory.

\section{Preliminaries}\label{section: Preliminaries}

We use finite element exterior calculus, for which standard references
include~\cite{ArnolFalkWinth10, Hiptm99}. In this section, we
establish the exterior calculus notations used in this paper and
recall results pertinent for the analysis of preconditioners.

\subsection{Sobolev spaces of exterior forms}  \label{sec:sobolev}

First, we set notations for $k$-forms in $n$-dimensions
($0 \le k \le n$).  The set of increasing multi-indices with $k$
components is denoted by
$\II_k = \{ \alpha = (\alpha_1, \ldots, \alpha_k): 1 \le \alpha_1 <
\alpha_2 < \cdots \alpha_k \le n \}$.  
For $\alpha \in \II^k$ and
$x = (x^1, x^2, \ldots, x^n),$ we abbreviate the elementary $k$-form
$dx^{\alpha_1} \wedge dx^{\alpha_2} \wedge \cdots \wedge
dx^{\alpha_k}$ to simply $dx^\alpha$.  
The space of $k$-forms on
$\RRR^n$ is denoted by
$\F k = \{ \sum_{\alpha \in \II_k} c_\alpha dx^\alpha: c_\alpha \in
\RRR\}$ and its dimension is $n_k \equiv (\begin{smallmatrix} n \\ k
\end{smallmatrix})$.
Let $H^s(\om)$ denote the standard Sobolev space on any open
$\om \subset \RRR^n$. The Sobolev space of exterior $k$-forms is
defined by 
\[
H^s(\om, \F k) = \bigg\{ w = \sum_{\alpha \in \II_k} w_\alpha(x) \; dx^\alpha 
:\; w_\alpha \in H^s(\om) \bigg\}.
\]
Its norm is given by 
\begin{equation}
  \label{eq:normcomp}
  \| w \|_{H^s(\om, \F k) }^2 
= \sum_{\alpha \in \II_k} \| w_\alpha \|_{H^s(\om)}^2.  
\end{equation}
The above notation scheme generalizes to analogously define other
spaces of forms like $L^2(\om, \F k)$, $C(\om, \F k)$, etc.  Thus
$\DD'(\om, \F k)$ denotes the space of $k$-forms whose components
$\varphi_{i_1 \cdots i_k}$ are distributions in $\DD'(\om)$ (where
$\DD(\om)$ is the space of smooth compactly supported test functions).
The inner product and norm of $L^2(\om, \F k)$ is denoted simply by
$(\cdot,\cdot) $ and $\| \cdot \|$, respectively. In either case the
form degree $k$ will be understood from context.

Let $d \equiv \dk$ denote the $k$th exterior derivative, e.g., when
applied to $w = w_\alpha dx^\alpha \in H^1(\om, \F k)$, the exterior
derivative $dw$ is given by 
\begin{equation}
  \label{eq:dw}
dw = \sum_{i=1}^n  \d_i w_\alpha \;dx^i \wedge dx^\alpha,
\end{equation}
where $\d_i w_\alpha$ is the usual $i$th partial derivative
$\d w_\alpha / \d x^i$ of the scalar multivariate function $w_\alpha$.
In three dimensions, $d^0$ generates the familiar gradient, $d^1$
generates curl, and $d^2$ generates the divergence operator. In four
dimensions, the exterior derivative has analogous interpretations,
which are worked out in detail later \rev{in \S\ref{sec:impl}.}

We are interested in the Sobolev spaces
\[
\Hd k = \{ w \in L^2(\om, \F k): \; dw \in L^2(\om, \F {k+1} ) \}
\]
normed by 
\[
\| w \|_{H(d, \om, \F k)}^2 = \| w \|^2 
 + 
\| d w\|^2.
\]
\rev{Note that when $k=n$, this space coincides with $L^2(\om, \F n)$
  (since $d=0$ then). When $n=3$, these spaces coincide with 
  the familiar three spaces $H^1(\om)$,
  $H(\text{curl}, \om)$ and $H(\div, \om)$ for $k=0,1,$ and $2$,
  respectively. For $n=4$, the corresponding four spaces are studied
  in detail in Section~\ref{sec:impl}.}

\subsection{Regular decomposition}

From now on, within this section, we tacitly assume that $\om$ is an
open bounded domain that is starlike with respect to a ball $B$, by
which we mean that for any $x \in \om$, the convex hull of $x$ and $B$
is contained in $\om$. This assumption implies that topology of $\om$
is trivial, i.e., $\om$ is homotopy equivalent to a ball, and that the
boundary of $\om$ is Lipschitz.  \rev{Under this assumption, certain
  regularized versions of homotopy operators of Poincar{\'{e}} are 
  constructed  in \cite{MitreMitreMonni08} (where its called
  averaged Cartan-like operators) and \cite{CostaMcInt09} (where its
  called regularized Poincar{\'{e}}-type integral operators). In the
  proof below, we shall follow the notation of \cite{CostaMcInt09} and
  denote these by $R_k$.  We use them to obtain a decomposition of
  $H(d,\om, \F k)$ into a more regular part and a remaining potential,
  as stated next.}

\begin{theorem}[Regular decomposition]
  \label{thm:regular}
  For each integer $1 \le k \le n$, 
  there is a $C_1>0$ and continuous linear maps
  \begin{align*}
    \SS & : H( d, \om, \F k) \to H^{1}(\om, \F k),
    &
    \PP & : H(d, \om, \F k) \to H^{1}(\om, \F {k-1})
  \end{align*}
  such that  \rev{for all  $w \in H(d, \om, \F k)$},
  \[
  w = \SS w + d\PP w 
  \]
  and 
  \[
  \| \SS w \|_{H^1(\om, \F k)} \le C_1 \| d w \|,
  \qquad
  \| \PP w \|_{H^1(\om, \F {k-1})} \le C_1 \| w \|_{H(d,\om, \F k)}.
  \]
\end{theorem}
\begin{proof}
  \rev{The regularized Poincar{\'{e}}-type integral operators of
    \cite[Corollary~3.4]{CostaMcInt09} are continuous linear operators
    $R_k: L^2(\om, \F k) \to H^1(\om, \F {k-1})$ for all
    $k=1, 2, \ldots, n-1$ satisfying $d R_k u + R_{k+1} d u = u$ for
    all $u \in H(d, \om, \F k)$. Moreover, the results of
    \cite{CostaMcInt09} when $k = n$ also yield $d R_n u = u$.
    Therefore, setting $\PP = R_k$ and $\SS = R_{k+1}$ the result
    follows for all $k=1, 2, \ldots, n-1$. It also follows for $k=n$
    once we set $\SS =0$ and $\PP = R_n$.}
\end{proof}

\rev{We note that regular decompositions were also given
in~\cite[Theorem~5.2]{HiptmLiZhou12} and \cite[Lemma~5]{DemloHiran14},
but their results do not state the first inequality of
Theorem~\ref{thm:regular}, which we need in the ensuing analysis.}

\subsection{Interpolation into finite element spaces}

Recall the well-known finite element
subspaces~\cite{ArnolFalkWinth10,Hiptm99} of $\Hd k$.  Let $P_r$
denote the space of polynomials in $n$ variables of degree at most
$r$,
$P_r\F k = \{ \sum_{\alpha \in \II_k} p_\alpha dx^\alpha: p_\alpha \in
P_r\},$ and let $P_r^- \F k \subseteq P_r \F k$, for all integers
$r\ge 1$, be as defined in~\cite[\S5.1.3]{ArnolFalkWinth10}. Let $\oh$
denote a geometrically conforming shape-regular simplicial finite
element mesh of $\om$.  Let $h$ denote the maximal mesh diameter
$h = \max_{K \in \oh} \diam(K)$.  To simplify technicalities, we
assume that the mesh $\oh$ is quasiuniform, so the diameter of every
element is bounded above and below by some fixed constant multiples
of~$h$.  The standard finite element subspaces of $\Hd k,$ indexed by
maximal mesh element diameter $h,$ are
$\Vhk = \{ v \in \Hd k: v|_K \in P_r^- \F k $ for all $n$-simplices
$K$ that are elements of the mesh $\oh\}$.  The Lagrange finite
element space $\Vh 0$ will play a special role in our discussions.  We
now introduce three operators that map various functions into $\Vhk$
that will be used in the sequel.

\rev{The first operator we need is the $L^2$ projection.  Identifying the
$n_k$-fold product of $\Vh 0$ as a subspace of $H^1(\om, \F k)$, we
denote it by $\Vnk.$  Let $Q_h = \Qhk : L^2(\om, \F k) \to \Vnk$ be
defined by
$
(\Qh z, v_h) = (z, v_h)
$
for all $v_h \in \Vnk$. Then, it follows from~\cite{BrambPasciStein02}
that for any $v \in H^1(\om, \F k)$,
\begin{equation}
  \label{eq:Qh}
  | \Qh v |_{H^1(\om, \F k)}
  + h^{-1} \| \Qh v - v \| 
  \prec |v|_{H^1(\om, \F k)}.
\end{equation}}
Here and throughout, we write $A \prec B$ to indicate that the
quantities $A$ and $B$ satisfy $A \le C B$ with a constant $C$ that is
independent of $h$ (but may depend on the shape regularity of $\oh$).

The next operator is the finite element interpolant~$\vPi_h \equiv \pihk$, often
called the {\em canonical interpolant}. 
A standard set of degrees of
freedom of $P_r^- \F k$ is well known (see~\cite[Theorem~5.5]{ArnolFalkWinth10} or
\cite{Hiptm99}). It  defines the canonical finite element interpolant
$\vPi_h$ in the usual way. Although the domain of $\vPi_h$ is often
viewed
as contained in a general (sufficiently regular) Sobolev space,
an important point of departure in this paper is to 
view $\vPi_h$ as a bounded linear
operator on {\em discrete spaces}, namely
\[
\pihk : \Vnk \to \Vhk.
\]
Lemma~\ref{lem:Pi} below provide continuity and approximation
estimates for $\vPi_h$ on the above domain.

Since $\vPi_h$ is, in general, unbounded on $\Hd k$, ideas to construct
bounded projectors into $\Vhk$ were proposed in~\cite{Schob08b} and
its antecedents. Such projectors are now well
known~\cite{ArnolFalkWinth10} by the name ``bounded cochain
projectors,'' Denoting them by $\Bhk$, we recall the standard
result~\cite[Theorem~5.9]{ArnolFalkWinth10} that
$\Bhk: L^2(\om, \F k) \to \Vhk$ is a bounded projection satisfying
\begin{subequations}  \label{eq:BddCochainProj}
  \begin{gather}
    \label{eq:BddCochainProj-1}
    \| w - \Bhk w \| 
     \prec h^s \| w \|_{H^s(\om, \F k)}
    \\   \label{eq:BddCochainProj-commute}
    d \Bhk = \Bh{k-1} d
  \end{gather}
\end{subequations}
\rev{for all $0 \le s \le r$.}

As a final note on the notation, {\em we will omit the superscript $(k)$
indicating the form degree from any notation} when no confusion can
arise. For example, just as $d$ abbreviates~$\dk$, we shall use
$\BB_h$ for $\Bhk$ when the form degree $k$ can be understood from context.

\section{The preconditioner}\label{section: The preconditioner}

\subsection{Definition}
\label{Definition:Preconditioner}

Let $\tau>0$ and let $A \equiv \A k : \Vhk \to \Vhk$ denote the
operator defined by
\begin{equation}
  \label{eq:A}
  (\A k u, v) = \tau (u,v) + ( du, dv)  
\end{equation}
for all $u, v \in \Vhk$.  Algebraic multigrid preconditioners for
$\A k$, for any form degree $k$, can be built by generalizing the
ideas in~\cite{HiptmXu07} and \cite{KolevVassi09} as we shall see in
this section.

The norm generated by $A$ is defined by $\| u \|_A = (A u, u)^{1/2}$.
Given two closed subspaces $V, W$ of $L^2$ and a linear operator
$R: V \to W$ we use $R^t : W \to V$ to denote its Hilbert adjoint
defined by $( R^t w,v) = (w, R v)$ for all $w \in W$ and $ v \in V$.
Let $d_h$ denote the restriction of $d$ on $\Vhk$, i.e.,
$d_h: \Vhk \to \Vh {k+1}$.  Then its adjoint
$d_h^t: \Vh{k+1} \to \Vhk$ is calculated by the above-mentioned
definition.

We define the preconditioner $ B \equiv \B k : \Vhk \to \Vhk$ for
$k \ge 1$ by induction on $k,$ supposing that for $k=0$, we are given
a good preconditioner $\B 0: \Vh 0 \to \Vh 0$,  i.e., there exists a 
$\beta \ge 1$ such that
\begin{equation}
  \label{eq:beta}
  \beta^{-1} (\B 0w,w) \le ( (\A 0)^{-1} w,w) \le \beta (\B 0  w, w)    
\end{equation}
for all $w \in \Vh 0$.  
Of course, we have in mind practically useful scenarios where $\beta$ 
is completely independent of (or very mildly dependent on)  $\tau$ and $h$.
The supposition of~\eqref{eq:beta} is justified since there are
good algebraic preconditioners~\cite{HensoYang02} for the Dirichlet
operator (arising from $\A 0$).  Then the $n_k$-fold product of
$\B 0,$ denoted by $\Bnk: \Vnk \to \Vnk$ preconditions
$\Ank: \Vnk \to \Vnk$, the $n_k$-fold product of $\A 0$.  Our aim is
 to use this to precondition $\B k$ for~$k>0$.

We need one more ingredient, the operator $D_h \equiv \Dhk : \Vhk \to \Vhk$
defined by
\[
(D_h u, v) = (h^{-2} + \tau) ( u,v) 
\]
for all $u, v \in \Vhk$. Finally, we define the preconditioner by
\begin{equation}
  \label{eq:B}
  B \equiv \B k = D_h^{-1} + \vPi_h \Bnk \vPi_h^t + \tau^{-1}  d_h \B {k-1} d_h^t
\end{equation}
for all $ 1 \le k \le n$.  Clearly, a practical implementation of this
preconditioner would need implementations of $\vPi_h$, $d_h$, and
$\Bnk$. The latter has been amply clarified in the literature (see
e.g.~\cite{HensoYang02}). In Section~\ref{sec:impl}, we will provide
more details on the implementation of $\vPi_h$ and $d_h$ when $n=4$.

Note that when the last term in~\eqref{eq:B} is recursively expanded, 
a simplification occurs, i.e., we have
\begin{align}
  \nonumber 
  d_h \B {k-1} d_h^t
  & = 
    \dhh{k-1} \bigg[
    (\Dh{k-1})^{-1} + \pih{k-1} \Bn{k-1} (\pih{k-1})^t 
  \\\nonumber 
  & \qquad\qquad + \tau^{-1}  
    \dhh{k-2} \B {k-2} (\dhh{k-2})^t
    \bigg] (\dhh{k-1})^t
\\ \label{eq:7}
  & = 
    \dhh{k-1}\left[
    (\Dh{k-1})^{-1} + \pih{k-1} \Bn {k-1} (\pih{k-1})^t
    \right] (\dhh{k-1})^t
\end{align}
because $d_h^{(k-1)}d_h^{(k-2)} =0$.  Thus the cost of applying the
preconditioner $\B k$ (ignoring the cost of inversion of $D_h$ and
the application of $\pihh$) is dominated by the cost of applying
$\Bn{k-1}$ and $\Bn {k}$, i.e., the cost of applying  $\B 0$
\begin{align*}
  n_k + n_{k-1} 
& = 
\begin{pmatrix}
  n+1 \\ k
\end{pmatrix}
\end{align*}
times. This also shows that an implementation of $B$ using only the
above-mentioned nonzero terms would be more efficient than simply 
implementing~\eqref{eq:B} recursively.

\subsection{Analysis}

\rev{We now proceed to prove a discrete version of the regular
  decomposition (as stated in Lemma~\ref{lem:lowerbdprm} below).  To
  this end, in addition to Theorem~\ref{thm:regular}, we need bounds
  on $\pihk$.  By viewing $\pihk$ as an operator acting on discrete
  spaces (as already mentioned earlier), we are able to use
  conclusions from scaling and finite dimensionality arguments for
  $k$-forms, such as the next two lemmas. We shall briefly display a
  proof of one of them using Euclidean coordinates.} Let $\hK$ denote
the unit $n$-simplex. There is an affine homeomorphism
$\Phi_K : \hK \to K$ for any $n$-simplex~$K$. Let $h_K = \diam(K)$.
Suppose $v$ is a $k$-form in $ L^2( K, \F k)$. Its pullback under
$\Phi_K$ is a $k$-form on $\hK$ denoted by $\Phi_K^* v$.

\begin{lemma}[Inverse inequality] \label{lem:inverse}
  For all $v_h \in \Vhk$, 
  we have 
  $
  \| d v_h \| \prec h^{-1} \| v_h \|.
  $
\end{lemma}

\begin{lemma}[Scaling of pullback] \label{lem:scale}
  For all $v \in L^2(K, \F k)$, 
  \begin{align}
    \label{eq:scal2}
    \| \Phi_K^* v \|_{L^2(\hK, \F k)}^2
             & \prec  h_K^{2k-n } \| v \|_{L^2(K, \F k)}^2 \prec 
               \| \Phi_K^* v \|_{L^2(\hK, \F k)}^2
\intertext{and for all $v \in H^1(K, \F k)$, }
    \label{eq:scal3}
    | \Phi_K^* v |_{H^1(\hK, \F k)}^2
             & \prec  h_K^{2 + 2 k - n} | v |_{H^1(K, \F k)}^2
               \prec     | \Phi_K^* v |_{H^1(\hK, \F k)}^2.
  \end{align}  
\end{lemma}
\begin{proof}
  Let $v = \sum_{\alpha \in \II_k} v_\alpha dx^\alpha$. 
  Its  pullback $\hat v = \Phi_K^* v$ when expanded in elementary form
  basis at any $\hat x \in \hK,$  takes the form
  \begin{align*}
    \hat v(\hat x)
    & = \sum_{\alpha \in \II_k} 
      \sum_{i_1=1}^n\cdots \sum_{i_k=1}^n
      v_{\alpha_1, \ldots, \alpha_k} (\Phi_K^{-1} \hat x)\, 
      \frac{\d x^{\alpha_1}}{\d \hat x^{i_1}}
      \frac{\d x^{\alpha_2}}{\d \hat x^{i_2}}
      \cdots 
      \frac{\d x^{\alpha_k}}{\d \hat x^{i_k}}
      \, 
      d\hat x^{i_1} \wedge  d \hat x^{i_2} \wedge \cdots \wedge d \hat x^{i_k}.
  \end{align*}
  \rev{Note that 
    $\| \d x^{\alpha_l}/ \d \hat
    x^{i_l}\|_{L^\infty(\hK)} \prec h_K.$} 
  To prove~\eqref{eq:scal3}, 
  applying~\eqref{eq:normcomp} but with
  norm replaced by seminorm,
  \begin{align*}
    | \hat v|_{H^1(\hK, \F k)}^2 
    & = \sum_{\beta \in \II_k}| \hat v_\beta|_{H^1( \hK, \F k)}^2
    \\
    & \prec 
      \sum_{\alpha \in \II_k} 
      \sum_{i_1=1}^n\cdots \sum_{i_k=1}^n
      \sum_{l=1}^n 
      \left\| 
      \frac{\d}{\d \hat x^l}\left( v_{\alpha_1, \ldots, \alpha_k} (\Phi_K^{-1} \hat x)\, 
      \frac{\d x^{\alpha_1}}{\d \hat x^{i_1}}
      \frac{\d x^{\alpha_2}}{\d \hat x^{i_2}}
      \cdots 
      \frac{\d x^{\alpha_k}}{\d \hat x^{i_k}}\right)
      \right\|^2_{L^2(\hK)}
    \\
    & \prec
    \sum_{\alpha \in \II_k} 
      \sum_{i_1=1}^n\cdots \sum_{i_k=1}^n
      \sum_{l=1}^n 
      \sum_{m=1}^n 
      \left\| 
      \frac{\d x^m}{\d \hat x^l}
      \d_m
      v_{\alpha_1, \ldots, \alpha_k}
      \frac{\d x^{\alpha_1}}{\d \hat x^{i_1}}
      \frac{\d x^{\alpha_2}}{\d \hat x^{i_2}}
      \cdots 
      \frac{\d x^{\alpha_k}}{\d \hat x^{i_k}}
      \right\|^2_{L^2(K)} \frac{ |\hat K| }{ | K|} 
    \\
    & \prec h_K^{2+ 2k}   \frac{ |\hat K| }{ | K|} 
      \sum_{\alpha \in \II_k}        \sum_{m=1}^n \| \d_m v_\alpha \|^2_{\rev{L^2(K)}}
      \;\prec\; h_K^{2 + 2k -  n} | v |_{H^1( \hK, \F k)}^2.
  \end{align*}
  \rev{The reverse inequality can be established by considering the inverse
  map $\Phi_K^{-1}$. The inequalities of~\eqref{eq:scal2} are proved
  similarly.}
\end{proof}

\begin{lemma}
  \label{lem:Pi}
  For all $v_h \in \Vnk$
  \begin{align}
    \label{eq:Pi-0}
    \| \vPi_h v_h \| 
    & \prec \| v_h \| 
      \\
    \label{eq:Pi-1}
    \| \vPi_h v_h - v_h \| 
    & \prec   h\| v_h \|_{H^1(\om, \F k)}
    \\
    \label{eq:Pi-2}
    \| d \vPi_h v_h \|
    & \prec \| v_h \|_{H^1(\om, \F k)}.
  \end{align}    
\end{lemma}
\begin{proof}
  Let
  $\vPi_K : P_r\F k (K) \to P_r^-\F k (K)$ be the canonical
  interpolant on $K$, i.e., $\vPi_K v = (\vpi_h v)|_K$ for any
  $v \in \Vnk$. Recall that any $v$ in $ \Vnk \subset H^1(\om, \F k)$,
  when restricted to $K$,  lies in $P_r \F k$. It is 
  easy to check that for any $v \in P_r\F k$, 
  $\Phi_K^* \vPi_K v  = \vPi_{\hK} \Phi_K^* v.$
  Since $\vPi_{\hK} : P_r\F k \to P_r^-\F k (\hK)$ is a linear map
  between finite dimensional spaces, it is bounded. Using
  Lemma~\ref{lem:scale}, we have
  \begin{align*}
  \| \vPi_h v \|^2_{L^2(K, \F k)} 
    & \prec h^{n-2k} \| \Phi_K^* \vPi v \|_{L^2(\hK, \F k)}^2
      = h^{n-2k} \| \vPi_{\hK} \Phi_K^* v \|_{L^2(\hK, \F k)}^2
    \\
    & \prec h^{n-2k} \| \Phi_K^* v  \|_{L^2(\hK, \F k)}^2
      \prec h^{n-2k} h^{2k-n} \| v  \|_{L^2(K, \F k)}^2.
  \end{align*}
  When summed over all $K \in \oh$, this proves~\eqref{eq:Pi-0}.
  
  To prove~\eqref{eq:Pi-1}, we note that $ c - \vPi_{\hK} c = 0 $ for
  any constant function $c$. Hence choosing $c$ to be the mean value
  of $\Phi_K^* v$ on $\hK$,
  \begin{align*}
    \| \vPi_h v - v \|_{L^2(K, \F k)}^2
    & \prec h^{n-2 k} \| (\vPi_{\hK} - I) (\Phi_K^* v - c) \|_{L^2(\hK, \F k)}^2
     \\
    & \prec h^{n-2 k} \| \Phi_K^* v - c \|_{L^2(\hK, \F k)}^2
     \prec h^{n-2 k}  | \Phi_K^* v|_{H^1(\hK, \F k)}^2
    \\
    &\prec   h^{n-2 k} h^{2 + 2k - n}  | v|_{H^1(\rev{K,} \F k)}^2
  \end{align*}
  where we have again used Lemma~\ref{lem:scale}.  Summing over all
  elements, this proves~\eqref{eq:Pi-1}.

  Finally to prove~\eqref{eq:Pi-2}, we note that the canonical
  interpolant commutes with the exterior derivative when applied to
  smooth functions. In particular, on any $v \in P_r \F k (K)$, 
  we have $d \pih k v|_K = \pih {k+1} d v|_K$. Hence, using the
  already established~\eqref{eq:Pi-0}, 
  \[
  \| d \pihh v \|_{L^2(K, \F k)}^2 \prec \| d v \|_{L^2(K, \F k)}^2 
  \prec | v |_{H^1(K, \F k)}^2.
  \]
  Summing over all elements, this proves~\eqref{eq:Pi-2}.
\end{proof}

\begin{lemma}[Stable decomposition]  \label{lem:lowerbdprm}
  For any $u_h \in \Vhk$, there are functions $s_h \in \Vhk$, 
  $z_h \in \Vnk$, and $p_h \in \Vh{k-1}$ such that 
  \begin{equation}
    \label{eq:2}
  u_h = s_h + \vPi_h z_h + d p_h    
  \end{equation}
  and 
  \begin{equation}
  \label{eq:3}    
  \begin{aligned}
    (h^{-2} + \tau) \| s_h  \|^2
    &+ \tau \| p_h \|^2
    +
    \tau \| d p_h \|^2 
    \\
    & +  \tau \| z_h    \|^2 +  |z_h |_{H^1(\om, \F k)}^2
    & \prec\; \tau \| u_h \|^2 + (1+ \tau) \| du_h\|^2.
  \end{aligned}
\end{equation}
\end{lemma}
\begin{proof}
  We apply Theorem~\ref{thm:regular}  to $u_h \in \Vhk \subset
  \Hd k$ to
  obtain
  \begin{gather}
    \label{eq:1}
    u_h = z + d p, 
    \\
    \label{eq:4}
    \| z \|_{H^1(\om, \F k)} \le C_1 \| d u_h \|, 
    \qquad \| p    \|_{H^1(\om, \F {k-1})}
    \le C_1  \| u_h \|_{H(d, \om,\F k)},
  \end{gather}
  where $z= \SS u$ and $p= \PP u$. \rev{Now let $z_h = \Qhk z \in \Vnk$.}
  Applying $\Bhk$ to both
  sides of~\eqref{eq:1} and using~\eqref{eq:BddCochainProj},
  \[
  u_h = \Bhh z + d \Bhh p.
  \]
  Then \eqref{eq:2} follows with 
  \[
  s_h = \Bhh z - \pihh z_h, \quad 
  p_h = \Bhh p
  \]
  and it only remains to prove the estimate~\eqref{eq:3}. 
  
  Observe that
  \begin{align*}
    \| z_h \|_{\A 0}^2 
    & = \tau \| z_h \|^2 + \| d^0 z_h \|^2 
      \prec \tau \| z \|^2 + | z|_{H^1(\om, \F k)}^2
      && \text{by \eqref{eq:Qh}}
    \\
    & \prec (1 + \tau) \,\rev{\| d u_h \|^2}
      && \text{by \eqref{eq:4}}.
    \\
    \tau \| p_h \|^2
    & = \tau \|\BB_h p\|^2   \le \tau \|p\|^2
      && \text{by \eqref{eq:BddCochainProj}}
    \\
    & \prec \tau  \|u_h\|^2 + \tau \| d u_h\|^2
      && \text{by \eqref{eq:4}}.
    \\
    \| d p_h \|^2 
    & = \| d \BB_h p \|^2 = \| \BB_h d p \|^2  \prec \| d p \|^2 
      && \text{by \eqref{eq:BddCochainProj}}
    \\
    & \prec  \| u_h\|^2 + \| z\|^2 
      && \text{by \eqref{eq:1}}
    \\
    & \prec \| u_h \|^2 + \| d u_h \|^2
      && \text{by \eqref{eq:4}}.
    \\
    \| s_h \|^2
    & \le ( \| \Bhk z - z \| + \| z - z_h \| + \| z_h - \pihk z_h \| )^2
      \\
    & \prec  h^2 \| z \|_{H^1(\om, \F k)}^2 \prec h^2 \| d u_h \|^2
  \end{align*}
  \rev{by \eqref{eq:BddCochainProj}, \eqref{eq:Qh}
    and~\eqref{eq:Pi-1}}.  Inequality~\eqref{eq:3} follows by
  combining these estimates.
\end{proof}

With the above lemmas, we are ready to conclude the analysis.  The
basis for the analysis of auxiliary space preconditioners is the
standard ``fictitious space lemma'' (see e.g.,
\cite{HiptmXu07,Nepom07,Xu96}) which we state without proof below in a
form convenient for us. Suppose we want to precondition a self-adjoint
positive definite operator $\vL$ on a finite-dimensional Hilbert space
$V$ using
\begin{enumerate}
\item a selfadjointpositive
definite operator $S: V \to V$ whose inverse is easy to apply, 
\item two ``auxiliary'' Hilbert
spaces $\Vt_1$ and $\Vt_2$ and linear operators $\Rt_i: \Vt_i \to V$,
and 
\item two further selfadjointpositive definite
operators $\vLt_i : \Vt_i \to \Vt_i$ on the auxiliary spaces whose inverses are easy to apply.
\end{enumerate}
In this setting, the following result guides the preconditioner
design. \rev{Here, 
  we denote  norms generated by selfadjoint
  positive definite  operators
  in accordance with our prior notation scheme, 
  e.g., $\| w \|_{\vLt_i} = (\vLt_i w, w)_{\Vt_i}^{1/2}$.}

\bigskip

\begin{lemma}[Nepomnyaschikh lemma] \label{lem:N}
  Suppose there are positive constants $c_1, c_2, c_s >0$ such that for all
  $\vt_j \in \Vt_j,$ $j=1,2,$ and $v \in V$,
  \begin{equation}
    \label{eq:Rupper}
    \| \Rt_1 \vt_1\|_\vL \le c_1 \| \vt \|_{\vLt_1},
    \qquad 
    \| \Rt_2\vt_2\|_\vL \le c_2 \| \vt \|_{\vLt_2},
    \qquad 
    \| v \|_\vL \le c_s \| v\|_S.    
  \end{equation}
  Suppose also that given any $v \in V$ there are $s \in V$,
  $\vt_i \in \Vt_i$ such that $s + \Rt_1\vt_1 + \Rt_2 \vt_2 = v $ and
  \begin{align}
\label{eq:Rlower}
    \| s\|_S^2 + 
    \| \vt_1 \|_{\vLt_1}^2 + \| \vt_2 \|_{\vLt_2}^2 
    & 
      \le c_0^2 \| v \|_\vL^2.
  \end{align}
  Then
  $P = S^{-1} + \Rt_1 \vLt_1^{-1} \Rt_1^t + \Rt_2 \vLt_2^{-1} \Rt_2^t$
  preconditions $\vL$ and the spectrum of $P \vL$ is contained in the
  interval $ [c_0^{-2}, c_1^2 + c_2^2+c_s^2]$.
\end{lemma}

\bigskip

\begin{theorem}
  \label{thm:main}
  Let $0 < \tau \,\rev{\prec 1}$ and let
  $A$ and $B$ be defined by \eqref{eq:A} and~\eqref{eq:B},
  respectively.  Suppose~\eqref{eq:beta} holds. 
  Then for each $1 \le k \le n-1$, there is an
  $\alpha\ge 1$ independent of $h$ and $\tau$ such that spectral
  condition number of $BA$ satisfies
  \[
  \kappa( BA) \le \alpha^2 \beta^2.
  \]
\end{theorem}
\begin{proof}
  First, we analyze the preconditioner 
  \begin{equation}
    \label{eq:P}
  \Pk  = D_h^{-1} + 
  \vPi_h (\Ank)^{-1} \vPi_h^t + \tau^{-1}  d_h (\A {k-1})^{-1} d_h^t.
  \end{equation}
  For this, we apply Lemma~\ref{lem:N} with
  \begin{align*}
    V & = \Vhk, && \vL = \A k, && S = D_h,
    \\
    \Vt_1&  = \Vnk, && \vLt_1=\Ank,  && \Rt_1 = \pihk,
    \\
    \Vt_2 & = \Vh {k-1} && \vLt_2 = \tau \A {k-1} && \Rt_2 = \dhh{k-1}.
  \end{align*}
  \rev{Note that 
    $V$ and $\Vt_i$ are endowed with $L^2$ inner products as before,
    so e.g., $\| w \|_{\vLt_1}^2 = (\smash{\vLt_1} w, w) = (\smash{\Ank} w,w) = \tau
    \| w \|^2 + (d^0 w, d^0w)$.}
  We must verify the
  conditions~\eqref{eq:Rupper} and~\eqref{eq:Rlower} of the lemma.

  To verify~\eqref{eq:Rupper}, we use the following bounds, which hold
  for any  $z_h \in \Vt_1$, $p_h \in \Vt_2$, and $v_h \in V$:
  \begin{align}
    \nonumber
    \| \Rt_1 z_h \|_\vL^2
    & = \| \pihh  z_h \|_A^2
      = \tau \| \pihh  z_h\|^2 + \| d \pihh z_h \|^2 
    \\ \label{eq:12}
    & \prec \tau \| z_h \|^2 + | z_h |_{H^1(\om, \F k)}^2 
      = \| z_h \|_{\vLt_1}^2,
    \\ \nonumber
    \| \Rt_2 p_h \|_\vL^2 
    & = \| d p_h \|_A^2 = \tau \| d p_h \|^2 
      \le \tau \left( \| d p_h \|^2 + \tau \| p_h \|^2 \right)
      \\ \nonumber
    & \;\rev{\ensuremath{=}}\; \tau \| p_h \|_{A}^2 = \|p_h\|_{\vLt_2}^2,
    \\ \nonumber
    \| v_h \|_{\vL}^2 & = \tau \| v_h \|^2 + \| d v_h \|^2 
    \\ \label{eq:13}
    &  \prec ( h^{-2} + \tau) \| v_h \|^2 = \| v_h \|_{D_h}^2.
  \end{align}
  We have used the inverse inequality of Lemma~\ref{lem:inverse} in
  the last bound. With the above bounds, we have   verified~\eqref{eq:Rupper}.

  Next, to verify~\eqref{eq:Rlower}, we use Lemma~\ref{lem:lowerbdprm}
  to decompose any $u_h$ in $V$ into
  $u_h = s_h + \Rt_1 z_h + \Rt_2 p_h = s_h + \pihh z_h + d p_h$ and
  apply~\eqref{eq:3}.  Since $\tau \le 1$, \eqref{eq:3} implies
  \begin{align*}
    \| s_h \|_{D_h}^2 + \| p_h \|_{\vLt_2}^2 + \| z_h \|_{\vLt_1}^2 
    &  = 
    \| s_h \|_{D_h}^2 + \tau \| p_h \|_{A}^2 + \| z_h \|_{\vLt_1}^2 
    \\
    & \le
    \| s_h \|_{D_h}^2 +
    \tau ( \tau \|p_h \|^2 + \| d p_h\|^2)
    + \| z_h \|_{\vLt_1}^2 
    \\
    & \prec \tau \| u_h \|^2 + (1 + \tau) \| d u_h \|^2
    \\
    & \prec \| u_h \|_A^2.
  \end{align*}
  This verifies~\eqref{eq:Rlower}. 
  Thus Lemma~\ref{lem:N} yields the existence of an $\alphak \ge 1$
  (after overestimating the constants if necessary) such that 
  \begin{equation}
    \label{eq:5}
    \frac{1}{\alphak} ( \Pk v,v) \le ( (\A k)^{-1} v,v) \le 
    \alphak (\Pk v,v)    
  \end{equation}
  for all $v \in V$.

  To complete the proof, we use the quadratic form of $\Pk$ to
  estimate that of $B$.  For any $v \in V$,
  \begin{align*}
    (\B k v,v) 
    & = (D_h^{-1} v,v) + (\Bnk \rev{\pihh^t} v, \rev{\pihh^t} v) 
      + \tau^{-1} (\B {k-1}      d_h^t v, \rev{d_h^t} v)
    \\
    & = ( (\Dh k)^{-1} v,v) + (\Bnk \rev{\pihh^t} v, \rev{\pihh^t} v) 
      \\
    & + \tau^{-1} 
      \left[ 
      ( (\Dh {k-1})^{-1} d_h^t v, d_h^t v) 
      + 
      ( ( \Bn {k-1})^{-1} \rev{\pihh^t} d_h^t, \rev{\pihh^t} d_h^t v)
      \right]      
  \end{align*}
  where we have used~\eqref{eq:7}. Now, using~\eqref{eq:beta} and
  \eqref{eq:P}, \rev{and~\eqref{eq:5}},
  \begin{align*}
    (\B k v,v) 
    & \le 
      (D_h^{-1} v,v) +
      \beta  ((\Ank)^{-1} \rev{\pihh^t} v, \rev{\pihh^t} v) 
    \\
    & 
      + \tau^{-1}\left[ 
      ( (\Dh {k-1})^{-1} d_h^t v, d_h^t v) 
      + 
      \beta ( ( \An {k-1})^{-1} \rev{\pihh^t d_h^t v}, \rev{\pihh^t} d_h^t v)
      \right]
      \\
    & \le \beta \left[
      (D_h^{-1} v,v) +
        ((\Ank)^{-1} \rev{\pihh^t} v, \rev{\pihh^t} v) 
      + \tau^{-1} ( \Pkm d_h^t v, d_h^t v)\right]
      \\
    & \le \beta \left[
      (D_h^{-1} v,v) +
        ((\Ank)^{-1} \rev{\pihh^t} v, \rev{\pihh^t} v) 
      + \tau^{-1} \alphakm ( (\A {k-1})^{-1} d_h^t v, d_h^t v)\right]
    \\
    & \le \beta \alphakm (\Pk v,v)
      \le \beta \alphakm \alphak 
      ( (\A k)^{-1} v,v).
  \end{align*}
  Combining with a similarly provable lower inequality, we have 
  \[
  \beta^{-1}\alphakm^{-1} \alphak^{-1}
  (\B k v,v) \le ( (\A k)^{-1} v,v) 
  \le \beta\alphakm \alphak
  (\B k v,v)
  \]
  for all $v \in \Vhk$.  
\end{proof}

\subsection{A variant}

The preconditioner in~\eqref{eq:B} is a generalization of auxiliary
space preconditioner in the form given in~\cite{HiptmXu07}. An
auxiliary space preconditioner in a slightly different form was
proposed in~\cite{KolevVassi09,KolevVassi12}. It can also be extended
to higher dimensions as we now show. To define this variant, let $Q_0$
denote the projection satisfying
\[
Q_0 u \in d \Vh {k-1}: \qquad (Q_0 u, \kappa) = (u, \kappa)\quad
\text { for all } \kappa \in d \Vh{k-1}.
\]
Then $A_0 = Q_0 A|_{d V_h} : d \Vh {k-1} \to d \Vh {k-1}$ satisfies
\begin{equation}
  \label{eq:9}
  (A_0 \kappa_1, \kappa_2) = (A \kappa_1,\kappa_2) = \tau (\kappa_1, \kappa_2)
\end{equation}
for all $\kappa_1,\kappa_2 \in d \Vh{k-1}$ because $d\circ d=0$. Clearly
$A_0$ is invertible for all $\tau>0$. Let $B_0: d \Vh{k-1} \to d \Vh{k-1}$ be a
preconditioner for $A_0$, i.e., there is a $\beta_0\ge 1$ such that
\begin{equation}
  \label{eq:B0}
  \beta_0^{-1} (B_0 \kappa, \kappa) \le (A_0^{-1} \kappa, \kappa) 
  \le \beta_0 ( B_0 \kappa, \kappa)
\end{equation}
for all $\kappa$ in $ d\Vh{k-1}$.  Using $B_0$, we define our second
auxiliary space preconditioner $C : \Vh k \to \Vh k$ by
\begin{equation}
  \label{eq:B2}
  C = D_h^{-1} + \vPi_h \Bnk \vPi_h^t + B_0 Q_0.  
\end{equation}

Unlike~\eqref{eq:B}, some care is needed to design and
implement~$B_0$.  Consider, for instance, the case $B_0 = A_0^{-1}$.
Although it appears from~\eqref{eq:9} that $A_0^{-1}$ is simply the
inverse of a mass matrix scaled by $\tau^{-1}$, the difficulty is that
we usually do not have a basis for $d \Vh{k-1}$ in a typical
implementation. To compute the action of the last term
in~\eqref{eq:B2} on some $v \in \Vhk$, namely
$\kappa_1 = A_0^{-1} Q_0 v,$ we write $A_0\kappa_1 = Q_0 v$ and
apply~\eqref{eq:9} to observe that $\kappa_1$ solves
\begin{align}
\label{eq:10}
\tau  (\kappa_1, \kappa_2)
  & =  (v, \kappa_2) \qquad \text{ for all } \kappa_2 \in d \Vh{k-1}.
\end{align}
Since we do not have a basis for $d\Vh{k-1}$, we use potentials $p_1, p_2$
in $\Vh{k-1}$ (for which we do have a basis) to express
$\kappa_i = d p_i$. Then~\eqref{eq:10} implies that $p_1$ in
$\Vh{k-1}$ solves
\begin{equation}
  \label{eq:11}
\tau  (dp_1, d p_2)
   =  (v, d p_2) \qquad \text{ for all } p_2 \in \Vh{k-1}.  
\end{equation}
Even if these equations do not uniquely determine $p_1$, this approach
does lead to a practical algorithm because we only need $d p_1$ to
apply~\eqref{eq:B2}. Note that $p_1$ is determined only up to the
kernel of $d$, but $\kappa_1 = d p_1$ is uniquely determined. One
strategy to compute $d p_1$ is to apply $d$ after computing a solution
$p_1$ given by the pseudoinverse of the system
in~\eqref{eq:11}. Another is to use an iterative technique that
converges to one solution of~\eqref{eq:11}. One may also use a
combination of such strategies, such as a multilevel iteration with
smoothers that are convergent despite the singularity
in~\eqref{eq:11}, combined with a coarse-level solver obtained from a
pseudoinverse. For more details, the reader may consult
\cite{KolevVassi09,KolevVassi12} or the implementations
in~\cite{hypre, mfem}. Notwithstanding the complications in
implementation, the analysis is a straightforward application of the
previous results.

\begin{theorem} \label{thm:variant} %
  Let $0 < \tau \le 1$ and let $A$ and $C$ be defined by \eqref{eq:A}
  and~\eqref{eq:B2}, respectively.  Suppose~\eqref{eq:beta}
  and~\eqref{eq:B0} holds and let $\beta_1 = \max(\beta,\beta_0).$ Then
  for each $1 \le k \le n-1$, there is an $\alpha\ge 1$ independent of
  $h$ and $\tau$ such that spectral condition number of $CA$ satisfies
  \[
  \kappa( CA) \le \alpha^2 \beta_1^2.
  \]
\end{theorem}
\begin{proof}
  As in the proof of Theorem~\ref{thm:main}, we consider an
  intermediate
  $P = D_h^{-1} + \vPi_h (\Ank)^{-1} \vPi_h^t + A_0^{-1} Q_0$, which
  is the preconditioner of Lemma~\ref{lem:N} with
  \begin{align*}
    V & = \Vhk, && \vL = \A k, && S = D_h,
    \\
    \Vt_1&  = \Vnk, && \vLt_1=\Ank,  && \Rt_1 = \pihk,
    \\
    \Vt_2 & = d\Vh {k-1} && \vLt_2 = A_0 = \tau I, && \Rt_2 = I.
  \end{align*}
  We proceed to verify the conditions~\eqref{eq:Rupper}
  and~\eqref{eq:Rlower}.  The following estimates holds for all
  $z_h \in \Vt_1$, $\kappa_h \in \Vt_2$, and $v_h \in V$:
  \begin{align*}
    \| \Rt_1 z_h \|_\vL^2
    & \prec
      \| z_h \|_{\vLt_1}^2, && \text{ by~\eqref{eq:12},}
    \\
    \| \Rt_2 \kappa_h \|_\vL^2 
    & = \| \kappa_h \|_A^2 = \tau \| \kappa_h \|^2 
      = \| \kappa_h \|_{\vLt_2}^2,
    \\
    \| v_h \|_{\vL}^2 &  \prec  \| v_h \|_{D_h}^2, && \text{ by~\eqref{eq:13}}.
  \end{align*}
  Hence we have verified~\eqref{eq:Rupper}.  To
  verify~\eqref{eq:Rlower}, as before, we use
  Lemma~\ref{lem:lowerbdprm} to decompose any $u_h$ in $V$ into
  $u_h = s_h + \Rt_1 z_h + \Rt_2 \kappa_h = s_h + \pihh z_h +
  \kappa_h$, where $\kappa_h = d p_h \in \Vt_2$, and
  apply~\eqref{eq:3} to get 
  \begin{align*}
    \| s_h \|_{D_h}^2 +
    \tau \| \kappa_h \|^2
    + \| z_h \|_{\vLt_1}^2 \prec \| u_h \|_{A}^2.
  \end{align*}
  where we have used the assumption that $\tau$ is bounded.  This
  verifies~\eqref{eq:Rlower}.  Thus Lemma~\ref{lem:N} yields the
  existence of an $\alphak \ge 1$ (after overestimating the constants
  if necessary) such that
  \[
    \alphak^{-1} ( P v,v) \le ( A^{-1} v,v) \le 
    \alphak (P v,v)    
  \]
  for all $v \in V.$

  To complete the proof, observe that for any $v \in V$, 
  \begin{align*}
    (C v,v) 
    & = (D_h^{-1} v,v) + (\Bnk \rev{\pihh^t} v, \rev{\pihh^t} v) 
      + (B_0 Q_0 v, Q_0 v)
    \\
    & \le (D_h^{-1} v,v)
      + \beta ( (\Ank)^{-1} \rev{\pihh^t} v, \rev{\pihh^t} v) 
      + 
      \beta_0 (A_0^{-1} Q_0 v, Q_0 v)
      \\
    & \le \beta_1 (P v,v) \le \beta_1 \alpha_k (A^{-1} v,v).
  \end{align*}
  Together with a similarly provable other-side bound, we have
  $ \beta_1^{-1} \alphak^{-1} (C v,v)$ $ \le $ $( A^{-1} v,v)$ $ \le $
  $\beta_1\alphak (C v,v)$ for all $v \in V$.
\end{proof}

\section{Implementation in 4 dimensions} \label{sec:impl}

In this section, we detail the implementation of the building blocks
of the preconditioner in four dimensions. These details form the basis
for our publicly available implementation of the preconditioner in the
MFEM package~\cite{mfem}. The vector space $\F k$ in general
dimensions is not usually implemented in finite element packages
(yet). Therefore, our approach is to view forms using elements (called
``proxies'' below) of the more standard vector spaces like $\RRR$,
$\RRR^4$, and the vector space of $4 \times 4$ skew symmetric
matrices~$\KKK$.

\subsection{Proxies of forms} \label{ssec:prox}

As already mentioned in \S\ref{sec:sobolev}, any 
$\varphi \in \F k$ has the basis expansion
\begin{equation} \label{eq:kform}
\varphi\; = \sum_{1 \le i_1 < \cdots < i_k \le 4} 
\varphi_{i_1 \cdots i_k} \; dx^{i_1} \wedge \cdots \wedge dx^{i_k}.
\end{equation}
Here the sum runs over all indices in $\II_k$ with four components.  The numbers $\varphi_{i_1 \cdots i_k}$, called the ``components'' or the ``coefficients'' of the form, are arranged into vectors or matrices that form ``proxies'' of $k$-forms, as defined below.

The proxy of a $k$-form $\varphi$ is denoted by $\p{k}{\varphi}$ and
is defined as follows.  In the case of a $0$-form $\varphi$, we set
$\p {0}{ \varphi} = \varphi$.  In the case of higher form degrees, we
use the components of $\varphi$ in~\eqref{eq:kform}, namely
$\varphi_i$ for 1-form, $\varphi_{ij}$ for 2-form, and $\varphi_{ijk}$
for 3-form, to define proxies:
\begin{align*}
  \p{1}{ \varphi } & = 
                      \begin{bmatrix}
                        \varphi_1 \\ \varphi_2 \\ \varphi_3 \\ \varphi_4 
                      \end{bmatrix},
                   & 
                     \p{2}{\varphi} & =
  \begin{bmatrix*}[r]
    0 & \varphi_{34} & -\varphi_{24} & \varphi_{23} 
    \\
    -\varphi_{34} & 0 & \varphi_{14} &  - \varphi_{13}
    \\
    \varphi_{24} & -\varphi_{14} & 0 & \varphi_{12} 
    \\
    -\varphi_{23} & \varphi_{13} & -\varphi_{12} & 0 
  \end{bmatrix*},
& 
\p{3}{\varphi} = 
\begin{bmatrix*}[r]
  \varphi_{234} 
  \\
  -\varphi_{134} 
  \\
  \varphi_{124}
  \\
  -\varphi_{123}
\end{bmatrix*}.
\end{align*}
Finally a $\varphi \in \F 4$ has only one component $\varphi_{1234}$,
so we set $\p{4}{\varphi} = \varphi_{1234}$.  Thus $\p {k} {\cdot}$
introduces a one-to-one onto correspondence from $\F k$ to
$\RRR, \RRR^4, \KKK, \RRR^4, $ and $\RRR$, for $k=0,1,2,3,$ and $4$,
respectively.

Some identities are expressed better using the permutation (or the
Levi-Civita) symbol $\veps_{i_1\,i_2\ldots\, i_n}$, whose definition
we recall briefly. For any $n$ and any indices $i_k$ in
$\{1,2,\ldots, n\}$, the value of $\veps_{i_1\,i_2\ldots\, i_n}$ is
zero when any two indices are equal. When the indices are distinct,
$i_1\,i_2\ldots i_n$ is a permutation of $1,2,\ldots, n$ and the value
of $\veps_{i_1\,i_2\ldots i_n}$ is set to the sign of the permutation.
It can be easily verified that the $(i,j)$th entry of the above-defined
skew symmetric matrix proxy of a $\varphi \in \F 2$ 
and 
the $i$th component of the proxy vector of a
$\varphi \in \F 3$ are given by
\begin{equation}
  \label{eq:prx23}
  \p{2} {\varphi}_{ij} = 
  \displaystyle\sum_{ 1 \le k < l \le 4} \veps_{ijkl} \varphi_{kl}, 
\qquad
  \p{3}{\varphi}_i = 
  \ds\sum_{ \rev{1 \le j < k < l \le 4}} \veps_{ijkl} \varphi_{jkl},
\end{equation}
where the sums run over increasing multi-indices in $\II_2$ and
$\II_3$, respectively.

Next, we define {\em two cross products} (both denoted by $\times$) in
the four-dimensional case. Recall that for any $u, v \in \RRR^3$, the
$i$th component of the standard cross product $u \times v$ is given by
$[u\times v]_i =\sum_{j,k=1}^3 \veps_{ijk} u_j v_k$. Analogously, we
define
\[
[u \times v]_{ij} = 
\sum_{k,l=1}^4 \veps_{ijkl} u_k v_l, \qquad 
u, v \in \RRR^4,
\]
i.e., this cross product of two 4-dimensional vectors yields a
skew-symmetric matrix.  We also define the cross product of two skew
symmetric matrices $\kappa, \eta$ in $\KKK$ by
\[
\kappa \times \eta = 
\sum_{1\le i<j\le 4}\; \sum_{1\le k<l\le 4} \veps_{ijkl} \kappa_{ij} \eta_{kl}, \qquad 
\kappa,\eta \in\KKK,
\]
i.e., result of the cross product of two matrices in $\KKK$ produces a
real number by a formula analogous to the cross product of
two-dimensional vectors (and the analogy is clear once we view the
skew-symmetric matrices as vectors in $\RRR^6$).

These operations, together with other standard multiplication
operations yield (after some elementary, albeit tedious calculations)
formulas for the wedge product and form action in terms of proxies, as
summarized in the next result. The standard products used below
include the scalar multiplication, the inner product of two vectors in
$\RRR^4$ (denoted by $\cdot$), the matrix-vector product of elements
in $\KKK$ with $\RRR^4$, and the Frobenius inner product between
matrices (denoted by $:$).

\begin{proposition} \label{prop:proxies_algebra}
  The following identities hold for the wedge product: 
  \begin{subequations}
    \begin{align}
      \p {k} { \eta \wedge \varphi } = 
      \p {k} { \varphi \wedge \eta}
      & = \p{0}{\varphi} \,\p {k} {\eta}, 
      && \varphi \in \F 0, \; \eta \in \F k,\;  k=0,\ldots,4, 
      \\
      -\p {2} { \eta \wedge \varphi} = 
      \p {2} { \varphi \wedge \eta}
      & = \p{1} {\varphi} \times \p{1} {\eta},
      && \varphi \in \F 1, \; \eta \in \F 1,
      \\
      \p{3}{ \eta \wedge \varphi} = 
      \p{3} {\varphi \wedge \eta}
      & = \p{2}{\eta} \,\p{1}{\varphi},
      && \varphi \in \F 1, \; \eta \in \F 2,
      \\
      -\p{4}{ \eta \wedge \varphi} = 
      \p{4}{\varphi \wedge \eta}
      & = \p{1}{\varphi} \cdot \p{3}{\eta},
      && \varphi \in \F 1, \; \eta \in \F 3,
      \\
      \p{4}{\eta \wedge \varphi} = 
      \p{4}{\varphi \wedge \eta}
      & = \p{2}{\varphi} \times \p{2}{\eta}
      && \varphi \in \F 2, \; \eta \in \F 2.
    \end{align}
  \end{subequations}
  The values of forms applied to vectors $u,v,w,z \in \RRR^4$ are
  given by
  \begin{subequations}
    \begin{align}
      \varphi(v) & = \p{1}{\varphi} \cdot v, 
      && \varphi \in \F 1,
      \\
      \varphi(u,v) & = \p{2} {\varphi} : (u \times v),
      && \varphi \in \F 2,
      \\
      \varphi(u,v,w) & = \det\left[ \p{3}{\varphi}, u, v, w\right],
      && \varphi \in \F 3,
      \\
      \varphi(u,v,w,z) & =  \p{4} {\varphi} \det[ u, v, w, z],
      && \varphi \in \F 4.
    \end{align}
  \end{subequations}
\end{proposition}

\subsection{The four  derivatives}  

Analogous to the three fundamental first order differential operators
($\grad, \curl,$ $\div$) in three dimensions, there are four first
order differential operators in four dimensions, which we denote by
\begin{equation}
  \label{eq:8}
  \grad, \;\Curl,\; \Div,\;
  \div.  
\end{equation}
The first and the last operators in~\eqref{eq:8} are
standard: For any $u \in \DD'( \om, \RRR)$ define
$\grad u \in \DD'(\om,\RRR^4)$ as the vector whose $i$th component is
$\d_i u$.  For any $v \in \DD'(\om, \RRR^4),$ we set
$\div v = \sum_{i=1}^4 \d_i v_i.$

Next, we define the four-dimensional $\Curl$ in a way that brings out
the analogies with the three-dimensional case.  Recall that the $i$th
component of the curl of a vector function $w$ in three dimensions can
be expressed as $[\curl w]_i = \sum_{j,k=1}^3 \veps_{ijk} \d_j w_k.$
Analogously, for any $w \in \DD'( \om, \RRR^4)$, we define $\Curl w$
as the matrix in $\KKK$ whose $(i,j)$th entry is defined by
\begin{equation}
  \label{eq:Curl}
  [\Curl w]_{ij} = \sum_{k,l=1}^4\veps_{ijkl} \d_k w_l
\end{equation}
i.e., 
\[
\Curl w = 
\begin{bmatrix}
  0 & \d_3 w_4 - \d_4 w_3 & \d_4 w_2 - \d_2 w_4 & \d_2 w_3 - \d_3 w_2
  \\
  \d_4 w_3 - \d_3 v_4 & 0 & \d_1w_4 - \d_4 w_1 & \d_3 v_1 - \d_1 w_3 
  \\
  \d_1 v_4 - \d_4 v_2 & \d_4 v_1 - \d_1 v_4 & 0 & \d_1 v_2 - \d_2 v_1 
  \\
  \d_3 v_2 - \d_2 v_3 & \d_1 v_3 - \d_3 v_1 & \d_2 v_1 - \d_1 v_2 & 0
\end{bmatrix}.
\]
Finally, the remaining operation $\Div$ acts on
$\kappa \in \DD'(\om, \KKK)$ and produces $\Div \kappa \in \RRR^4$ by
taking divergence row-wise, i.e.,
\begin{equation}
  \label{eq:Div}
  [\Div \kappa]_i = \sum_{j=1}^4 \d_j \kappa_{ij}.
\end{equation}
Note that the identities
$ \Curl (\grad u) = 0,$ $ \Div ( \Curl w) = 0,$ and
$\div ( \Div \kappa) = 0$ follow immediately from the above
definitions.

By connecting the inputs and outputs of the above-introduced four
differential operators to proxies of forms, we may understand them as
manifestations of exterior derivatives. In fact, the following diagram
commutes:
\[
\xymatrixcolsep{2.9pc}\xymatrixrowsep{3.5pc}
\xymatrix{
  \DD'(\om, \F 0) \ar[r]^{\dd{0}} \ar[d]|{\quad\p{0}{\,\cdot\,}} & 
  \DD'(\om, \F 1) \ar[r]^{\dd{1}} \ar[d]|{\quad\p{1}{\,\cdot\,}} & 
  \DD'(\om, \F 2) \ar[r]^{\dd{2}} \ar[d]|{\quad\p{2}{\,\cdot\,}} & 
  \DD'(\om, \F 3) \ar[r]^{\dd{3}} \ar[d]|{\quad\p{3}{\,\cdot\,}} & 
  \DD'(\om, \F 4) \ar[d]|{\quad\p{4}{\,\cdot\,}} \\
  \DD'(\om, \RRR)   \ar[r]^{\grad} & 
  \DD'(\om, \RRR^4) \ar[r]^{\Curl} & 
  \DD'(\om, \KKK)   \ar[r]^{\Div} & 
  \DD'(\om, \RRR^4) \ar[r]^{\div} & 
  \DD'(\om, \RRR). 
}
\]
This follows from the identities collected next, which can again be
proved by elementary calculations.

\begin{proposition}\label{prop:proxies_derivatives}
  The following identities hold:
  \begin{subequations}
    \label{eq:cdprx}
    \begin{align}
      \label{eq:cdprx0}
      \p{1}{ \dd{0}\varphi} & = \grad( \p{0} \varphi), 
                            &&  \varphi \in \DD'(\om, \F 0),
      \\
      \label{eq:cdprx1}
      \p{2}{ \dd{1} \varphi} & = \Curl (\p{1} \varphi),
                            &&  \varphi \in \DD'(\om, \F 1),
      \\
      \label{eq:cdprx2}
      \p{3}{ \dd{2}\varphi} & = \Div (\p{2} \varphi),
                            &&  \varphi \in \DD'(\om, \F 2),
      \\
      \label{eq:cdprx3}
      \p{4}{ \dd{3} \varphi} & = \div (\p{3} \varphi),
                            && \varphi \in \DD'(\om, \F 3).
    \end{align}
  \end{subequations}
\end{proposition}

In addition to the operator~$\Curl$, another curl operator deserves
mention because it fits in an alternate sequence of spaces in four
dimensions.  Let $\MMM$ denote the space of $4 \times 4$ matrices and
let $\skw m = ( m - m^t)/2$ for any $m \in \MMM$. Define the curl of a
skew-symmetric matrix, namely
$\curl: \DD'(\om, \KKK) \to \DD'(\om, \RRR^4),$ and an
antisymmetrization operator $K: \MMM \to \KKK$, by
\begin{equation}
  \label{eq:Kdefn}
  [\curl \omega ]_i = \sum_{k,l=1}^4 \veps_{ijkl} \d_j
  \omega_{kl},
\qquad
  [K m ]_{ij} =  \sum_{k,l=1}^4 \veps_{ijkl}  m_{kl}
\end{equation}
for any $\omega \in \DD'(\om, \KKK)$ and $m \in \MMM$. Also let the
gradient of a vector field,
$\Grad : \DD'(\om, \RRR^4 ) \to \DD'(\om, \MMM),$ be defined by
$[\Grad u]_{ij} = \d_j u_i$.  Now, analogous to the previously
discussed sequence, 
\[
\xymatrixcolsep{2.9pc}
\xymatrix{
  \DD'(\om, \RRR)   \ar[r]^{\grad} & 
  \DD'(\om, \RRR^4) \ar[r]^{\Curl} & 
  \DD'(\om, \KKK)   \ar[r]^{\Div} & 
  \DD'(\om, \RRR^4) \ar[r]^{\div} & 
  \DD'(\om, \RRR),
}
\]
we may study the following sequence with the newly defined $\curl$:
\[
\xymatrixcolsep{2.9pc}
\xymatrix{
  \DD'(\om, \RRR)   \ar[r]^{\grad} & 
  \DD'(\om, \RRR^4) \ar[r]^{\skw\Grad} & 
  \DD'(\om, \KKK)   \ar[r]^{\curl} & 
  \DD'(\om, \RRR^4) \ar[r]^{\div} & 
  \DD'(\om, \RRR).
}
\]
The properties of the second sequence can be derived from that of the
first using 
\begin{equation}
  \label{eq:14}
 K \skw \Grad u =  \Curl u, 
 \qquad 
 \curl \omega = \Div  K \omega 
\end{equation}
for all $u \in \DD'(\om, \RRR^4)$ and $\omega \in \DD'(\om, \KKK)$. In
particular, $\skw\Grad \grad =0$, $\curl \skw\Grad =0$ and
$\div \curl =0$.

\subsection{Sobolev spaces}

In view of the identities of Proposition~\ref{prop:proxies_derivatives}, the
spaces $\Hd k$ in four dimensions are identified to be the same as
\begin{align*}
  H(\grad, \om, \RRR) 
  & = \{ u \in L^2(\om, \RRR):\; \grad u \in L^2(\om, \RRR^4) \},
  \\
  H(\Curl, \om, \RRR^4)
  & = \{ v \in L^2(\om, \RRR^4): \;\Curl v \in L^2(\om, \KKK) \},
  \\
  H(\Div, \om, \KKK)
  & = \{ \kappa \in L^2(\om, \KKK):\; \Div \kappa \in L^2(\om, \RRR^4)\},
  \\
  H(\div, \om, \RRR^4)
  & = \{ q \in L^2(\om, \RRR^4):\; \div q \in L^2(\om, \RRR) \},
\end{align*}
for $k=0,1,2,3$, respectively. Also setting 
\[
  H(\curl, \om, \KKK)
  = \{ \omega \in L^2(\om, \KKK):\; \curl \omega \in L^2(\om, \RRR^4)\},
\]
we note that the operator $K$ defined in~\eqref{eq:Kdefn} yields a
one-to-one onto homeomorphism
$ K : H(\curl, \om, \KKK) \to H(\Div, \om, \KKK)$.

In addition to the usual Green's formula involving gradient and
divergence, one can derive other integration by parts formulae, which
also clarify the nature of traces in the new spaces.  Let
$\DD(\bar \om, \KKK)$ and $\DD(\bar\om, \RRR^4)$ denote the sets of
restrictions to $\om$ of functions in $\DD(\RRR^4, \KKK)$ and
$\DD(\RRR^4, \RRR^4)$, respectively. Let $\om$ have Lipschitz boundary
so that the unit outward normal $n$ on $\d\om$ is well defined a.e.
Then we can show that the traces $(n \times u)|_{\d\om}$ and
$(\omega n)|_{\d\om}$ have meaning for $u \in H(\Curl, \om, \RRR^4)$
and $H(\Div, \om, \KKK)$. More precisely, define
\[
(\tr{1}u)(\og) = \int_{\d\om} (n \times u) : \og, \qquad
(\tr{2}\og)(u) = \int_{\d\om} \rev{\og n \cdot u}
\]
for all $u \in \DD(\om, \RRR^4)$ and $\og\in \DD(\om, \KKK)$.

\begin{proposition}
  \label{prop:intbyparts}
  Suppose $\om$ has Lipschitz boundary. Then $\tr{1}$ and $\tr{2}$
  extend to continuous linear operators
  $ \tr{1} : H(\Curl, \om, \RRR^4) \to H(\curl,\om, \KKK)', $ and
  $ \tr{2} : H(\Div, \om, \KKK) \to H(\Curl,\om, \RRR^4)' $ satisfying
  \begin{align*}
    ( \tr{1} u)( \omega)
    & = \int_\om \Curl u : \omega  -\int_\om \rev{u \cdot \curl \omega}
    \\
    (\tr{2} \kappa)(u)
    & = \int_\om \Div \kappa \cdot u \; dx -\int_\om \kappa \times \Curl u
  \end{align*}
  for all $u \in H(\Curl, \om, \RRR^4)$,  $\kappa \in H(\Div, \om,
  \KKK),$ and $\omega \in H(\curl, \om, \KKK)$.
\end{proposition}
\begin{proof}
  For $u \in \DD(\om, \RRR^4)$ and $\og \in \DD(\om, \KKK)$,
  integrating by parts each term that makes up the products below and
  using the properties of $\veps$ to simplify the result, we derive
  \begin{equation}
    \label{eq:intbyparts1}
    \int_\om \Curl u : \og -\int_\om u \cdot \curl \og = \int_{\d\om} (n
    \times u) : \og.  
  \end{equation}
  Similarly, we also derive
\begin{equation}
  \label{eq:intbyparts2}
 \int_\om \Div \og \cdot u  -\int_\om \og \times \Curl u
 = \int_{\d\om} \og n \cdot u.
\end{equation}
Now viewing $H(\Curl, \om, \RRR^4)$, $H(\Div, \om, \KKK)$, and
$H(\curl, \om, \KKK)$ as graph spaces of $\Curl$, $\Div$, and $\curl$,
we apply the well-known extensions of classical density proofs to
graph spaces (see e.g., \cite{Jense04}) to conclude that
$\DD(\bar \om, \RRR^4)$ is dense in $H(\Curl, \om, \RRR^4)$ and that
$\DD(\bar \om, \KKK)$ is dense in $H(\curl, \om, \KKK)$ as well as
$H(\Div, \om, \KKK)$.  Hence the result follows
from~\eqref{eq:intbyparts1} and~\eqref{eq:intbyparts2}.
\end{proof}

The Sobolev spaces we have introduced above have their
analogues with essential boundary conditions:
\begin{align*}
H_0(\Curl, \om, \RRR^4)
  & = \{ v \in H(\Curl, \om, \RRR^4): \;\tr{1} v = 0 \}, 
\\
H_0(\Div, \om, \KKK) 
& = \{ \og \in H(\Div, \om, \KKK): \; \tr{2} \og = 0\}.
\end{align*}
The construction of the HX preconditioner for these spaces follows
along the same lines as before, now using standard preconditioners in
$H_0^1(\om)$. \rev{To highlight the changes required in the analysis,
  first, instead of the regularized {Poincar{\'{e}}} operator $R_k$
  (appearing in the proof of Theorem~\ref{thm:regular}) , we must now
  use the generalized {Bogovski{\u \i}} operator (see
  \cite{CostaMcInt09} or \cite[Theorem~1.5]{MitreMitreMonni08}) to get
  the appropriate regular decomposition with boundary conditions.  We
  then continue along the previous lines after replacing $\Qhk$ by the
  $L^2$ projection into $\Vhk \cap H_0^1(\om, \F k)$.  Note that the
  $H_0^1$-stability of this projection holds as remarked
  in~\cite[p.~153]{BrambPasciStein02}. Bounded cochain projectors
  preserving homogeneous boundary conditions are also
  known~\cite{ChrisWinth08}, so all the ingredients are available to
  generalize our analysis to the case of homogeneous essential
  boundary conditions.}

\subsection{Finite element spaces}

Let $T$ be a $4$-simplex with vertices $a_i$, $i=1,\ldots, 5$. Let
$\lambda_i$ denote its $i$th barycentric coordinate, i.e.,
$\lambda_i(x)$ is the unique affine function (of the Euclidean
coordinate~$x$ of points in $T$) that equals 1 at $a_i$ and equals $0$
at all the remaining vertices of~$T$. Let
\[
g_i = \grad \lambda_i \in \RRR^4, 
\qquad
g_{ij} = g_i \times g_j
 \in \KKK,
\qquad  
g_{ijk} = g_{ij} g_k \in \RRR^4.
\]
Let $f_{i_1,\ldots,i_k}$ denote the subsimplex of $T$ formed by the
convex hull of $a_{i_1}, \ldots, a_{i_k}$ for any $k=1,\ldots, 5$ and
let $\trg{k}{T}$ denote the set of all $k$-subsimplices of $T$.
To a $0$-subsimplex $f_i=a_i$ we associate the function $\lambda_i$ and to
other subsimplices $f_{ij}$, $f_{ijk}$ and $f_{ijkl}$, we associate, 
respectively, the following functions.
\begin{subequations}
\label{eq:whitney}
\begin{align}
  \lambda_{ij} 
  & = \lambda_i g_j - \lambda_j g_i,
  \\
  \lambda_{ijk}
  & = \lambda_i g_{jk} - \lambda_j g_{ik} + \lambda_k g_{ij},
  \\
  \lambda_{ijkl}
  & = \lambda_i g_{jkl} - \lambda_j g_{ikl} + \lambda_k g_{ijl} 
    - \lambda_l g_{ijk}.
\end{align}  
\end{subequations}
Note that these expressions depend on the ordering of the vertices and
on $T$. When such dependence is to be made explicit, we write the
function associated to any $f_{i_1,\ldots, i_k}\in \trg{k}{T}$, namely
$\lambda_{i_1,\ldots, i_k}$, as $\lambda_{a_{i_1},\ldots, a_{i_k}}^T$
or $\lambda_{a(f)}^T$ where $a(f)= (a_{i_1},\ldots, a_{i_k}).$

We implemented the lowest order polynomial space
$ P_1^{-,(k)}(T) = \spn \{ \lambda_{a(f)}: f \in \trg{k}{T}\}$ for
$k=0, 1,2$ and $3$. Using Propositions~\ref{prop:proxies_algebra}
and~\ref{prop:proxies_derivatives}, these spaces may be immediately
recognized as the space of proxies of the {\em Whitney basis}
\rev{\cite{ArnolFalkWinth06, 
ArnolFalkWinth09,Whitn57}} for $P_1^-\F k$.  To construct the
global finite element spaces, we consider the set of all
$k$-subsimplices of the simplicial mesh $\oh$, denoted by
$\trg{k}{\oh}$. An element $f$ of $\trg{k}{\oh}$ is in the set
$\trg{k}{T_j}$ for one or more mesh elements $T_1, \ldots, T_{n_f}$ in
$\oh$.  To each $f \in \trg{k}{\oh}$, we associate an {\em ordered}
set of its vertices $a(f)$. The ordering fixes a global orientation
of~$f$ independently of $T_j$.  Let $\lambda_{f}$, for each
$f \in \trg{k}{\oh}$, be the function that vanishes on all elements of
the mesh except $T_1, \ldots, T_{n_f}$ where its values are given by
$\lambda_f|_{T_j} = \lambda_{a(f)}^{T_j}.$ These functions define the
global finite space by 
\[
\Vp k = \spn \{\lambda_f: f \in \trg{k}{\oh}\}
\]
for each $k=0,1,2,3$. One can easily show that
$ \Vp 1 \subseteq H(\Curl, \om, \RRR^4), $ and
$ \Vp 2 \subseteq H(\Div, \om, \KKK),$ either directly integrating by
parts using Proposition~\ref{prop:intbyparts} on each mesh element, or
by observing that $\Vp{k}$ consists of all proxies of $\Vhk$
(when~$r=1$) and recalling \cite{ArnolFalkWinth06} that
$\Vhk \subseteq \Hd k$. Of course, we also have
$\Vp 0 \subseteq H(\grad, \om, \RRR^4)$ and
$\Vp 3 \subseteq H(\div, \om, \RRR^4)$. 

Our actual implementation uses an alternate, but equivalent technique
that proceeds by implementing the expressions in~\eqref{eq:whitney}
only on the unit $4$-simplex and then mapping the basis functions to
each mesh simplex appropriately (see the code in \cite{mfem} for more
details).

\subsection{Finite element interpolant}

The implementation of the HX preconditioner for
$\Vp 1 \subseteq H(\Curl,\om, \RR^4)$,
$\Vp 2 \subseteq H(\Div, \om, \KKK)$, and
$\Vp 3 \subseteq H(\div, \om, \RRR^4)$ requires us to implement
canonical finite element interpolants $\piCurl, \piDiv, $ and
$\pidiv$, into these spaces, respectively. Since the last one is
standard, we only describe the first two.

Let $T$ be a $4$-simplex with vertices $a_i$ and let $u: T \to \RRR^4$
be a smooth vector function. Let $e_{ij}$ denote the segment
connecting $a_i$ and $a_j$.  Then $\piCurl u|_T$ is the unique function
in $P_1^{-,(1)}(T)$ satisfying $\sigma_{ij} ( u - \piCurl u) =0$ for
every edge $e_{ij}$ of $T$, where 
\[
\sigma_{ij}(u) = \frac{1}{|e_{ij}|} \int_{e_{ij}} u \cdot (a_i - a_j) 
\]
and $|e_{ij}|$ denotes the length of the edge $e_{ij}$.

Next, let $\og: T \to \KKK$ be a smooth function and let $f_{ijk}$
denote the triangle formed by the convex hull of $a_i, a_j$ and $a_k$.
Then $\piDiv \og|_T$ is the unique function in $P_1^{-,(2)}(T)$
satisfying $\sigma_{ijk} ( \og - \piDiv \og) =0$ for all
$2$-subsimplices $f_{ijk}$ of $T$, where
\[
\sigma_{ijk}(\og) = \frac{1}{|f_{ijk}|}
\iint_{f_{ijk}} \og : (a_j - a_i) \times (a_k - a_i) 
\]
and $|f_{ijk}|$ denotes the area of the triangle $f_{ijk}$.

To compute the preconditioner action, we need to apply $\piCurl$ to
functions $u: \om \to \RRR^4$ whose components are in the lowest order
Lagrange finite element space. Then defining $\piCurl u|_T$ for each
$T$ in $\oh$ as above, the continuity of components of $u$ imply that
the resulting global function $\piCurl u$ in $\Vp 1.$ Similarly, when
$\og$ has components in the Lagrange finite element space, $\piDiv
\og$ is in $\Vp 2.$

The unisolvency of these degrees of freedom follow
from~\cite[Theorem~5.5]{ArnolFalkWinth10} after identifying the
degrees of freedom given there (for $k=1,2$) in terms of our proxies
using Proposition~\ref{prop:proxies_algebra}. In particular, $\piCurl$
and $\piDiv$ can be viewed as proxies of $\pihk$ for $k=1$ and~$2$ in
four dimensions.

\section{Numerical results}   \label{sec:numer}

In this section, we report the results of numerical experiments
obtained using our implementation of the preconditioners in 4D.  We
implemented the lowest order finite element subspaces of $H^1(\om)$,
$H(\Curl, \om)$, $H(\Div, \om)$ and $H(\div, \om)$ on general
unstructured (conforming) meshes of 4-simplices. The preconditioners
were built atop this discretization. Below we will perform
verification of the discretization as well as report on the
performance of the preconditioners.

\subsection{Convergence studies}
\label{SubSection:ConvergenceStudies}

In the first series of examples, we fix $\om = (0,1)^4$ and solve
the linear systems arising from the lowest order finite element
discretization of the following problem: Find $u \in H(d, \om, \F k),$
such that
\begin{align}
\label{eq:numerprob}
  (u,v) + ( du, dv) = \langle F, v \rangle \qquad \text{for all } v \in H(d, \om, \F k),
\end{align}
where $F \in H(d, \om, \F k)'$ is a bounded linear functional given
below for each $k=0,1,2,3$.  The domain $\om$ was initially subdivided
into a mesh $\oh$ of 96 4-simplices of uniform size (see also
\cite{NeumuellerSteinbach11}).  Afterwards we apply successive
refinement based on the algorithm of Freudenthal (see \cite{Bey00,
  Freudenthal42, NeumuellerSteinbach11} for more details).  The
arising linear systems are solved using preconditioned conjugate
gradient iterations, where the preconditioner is set to the ones given
in \S\ref{Definition:Preconditioner} for each $k$. In all the
presented experiments set the smoother $D_h$ by three steps of a
Chebyshev smoother with respect to the operator $A$. We iterate until
a relative residual error reduction of $10^{-6}$ is obtained.  Under
these numerical settings, we study two types of convergence, namely
the convergence rates of the lowest order 4D discretizations, and the
iterative convergence of the preconditioned conjugate gradient
iterations.

To establish a baseline, we start with $0$-forms, i.e. $d =
\grad$. The $F$ in~\eqref{eq:numerprob} is set so that the exact
solution is 
\begin{align*}
  u(x) = \cos(\pi x_1) \cos(\pi x_2) \cos(\pi x_3) \cos(\pi x_4).
\end{align*}
The $L^2(\om)$ distance between this $u$ and the computed solution
$u_h$ in the 4D lowest order Lagrange finite element space is reported
in one of the columns of Table~\ref{Table:Experiment1:H(grad)}.
Clearly the observed convergence rate is close to two, the best
possible rate for this approximation space.  For solving the linear
systems we set the preconditioner to the algebraic multigrid
preconditioner BoomerAMG of the hypre package~\cite{HensoYang02}.  The
iteration counts reported in the last column of the same table show
small iteration numbers with small growth.  Recall that one of the
basic assumptions in the auxiliary space preconditioner construction
is that we have a good preconditioner for the Laplacian. Therefore
this report of the performance of BoomerAMG in 4D gives us a measure
of how well this baseline assumption is verified in practice.

For $1$-forms, i.e. $d=\Curl$, we set an $F$ in \eqref{eq:numerprob}
that yields the exact solution
\begin{align*}
 u(x) = [ s_1 c_2 c_3 c_4, -c_1 s_2 c_3 c_4, c_1 c_2 s_3 c_4, -c_1 c_2 c_3 s_4 ]^\top,
\end{align*}
where  $c_i = \cos(\pi x_i)$ and $s_i = \sin(\pi x_i)$ for
$i=1,\ldots,4$. In Table \ref{Table:Experiment1:H(Curl)} we summarize
the convergence results for the lowest order finite elements, i.e.,
edge-elements in 4D.  Here, we again observe a convergence rate close
to the theoretically expected rate of one. For solving the linear
system we use the proposed preconditioner given in Subsection
\ref{Definition:Preconditioner}. Here we obtain small iteration counts. But
observe that they are slightly increasing. We believe this is due to
the fact that the BoomerAMG's  performance (reported in the
previous table) is not strictly uniform.

When considering $2$-forms, i.e. $d=\Div$ we use the manufactured
solution
\begin{align*}
 u(x) = \begin{bmatrix}
  0 & c_1 c_2 s_3 s_4 & -c_1 s_2 c_3 s_4 & c_1 s_2 s_3 c_4 \\
  -c_1 c_2 s_3 s_4 & 0 & s_1 c_2 c_3 s_4 & -s_1 c_2 s_3 c_4 \\
  c_1 s_2 c_3 s_4 & -s_1 c_2 c_3 s_4 & 0 & s_1 s_2 c_3 c_4 \\
  -c_1 s_2 s_3 c_4 & s_1 c_2 s_3 c_4 & -s_1 s_2 c_3 c_4 & 0
 \end{bmatrix},
\end{align*}
where $c_i, s_i$ are as above.  Using the lowest order finite elements
for $2$-forms in 4D, we observe in Table
\ref{Table:Experiment1:H(Div)} the optimal convergence rate of
one. Moreover the preconditioner given in Subsection
\ref{Definition:Preconditioner} leads to similar iteration numbers as
in the previous example.

\begin{footnotesize}
\begin{table}
  \begin{center}
  \begin{tabular}{c|r|r r|cc|c}
    level & cores & elements & dof & $||u - u_h||_{L^2(\Omega)}$ & eoc & iter \\ \hline
    0 &     16 &     96 &     25 & 2.21816E-1 & - & 6\\
    1 &     16 &    1 536 &     169 & 1.64804E-1 & 0.43 & 11\\
    2 &     128 &    24 576 &    1 681 & 7.38929E-2 & 1.16 & 15\\
    3 &    16 384 &    393 216 &    21 025 & 2.63863E-2 & 1.49 & 18\\
    4 &    32 768 &   6 291 456 &    297 025 & 7.77695E-3 & 1.76 & 20\\
    5 &    32 768 &   100 663 296 &   4 464 769 & 2.06687E-3 & 1.91 & 22\\
    6 &    32 768 &  1 610 612 736 &   69 239 041 & 5.26637E-4 & 1.97 & 25\\ \hline
  \end{tabular}
  \end{center}   
\caption{Convergence result and iteration numbers for $H(\grad, \om, \RRR)$.}
\label{Table:Experiment1:H(grad)}
\end{table}
\begin{table}
  \begin{center}
  \begin{tabular}{c|r|r r|cc|c}
  level & cores & elements & dof & $||u - u_h||_{L^2(\Omega)}$ & eoc & iter \\ \hline
0 &     16 &     96 &     144 & 3.77018E-1 & - & 10 \\
1 &     16 &    1 536 &    1 512 & 3.05056E-1 & 0.31 & 15 \\
2 &     128 &    24 576 &    19 344 & 1.92098E-1 & 0.67 & 18 \\
3 &    16 384 &    393 216 &    276 000 & 1.10859E-1 & 0.79 & 22 \\
4 &    32 768 &   6 291 456 &   4 167 744 & 5.94185E-2 & 0.90 & 25 \\
5 &    32 768 &   100 663 296 &   64 774 272 & 3.05043E-2 & 0.96 & 28 \\
6 &    32 768 &  1 610 612 736 &  1 021 411 584 & 1.53829E-2 & 0.99 & 35 \\ \hline
  \end{tabular}
  \end{center}
\caption{Convergence result and iteration numbers for $H(\Curl, \om, \RRR^4)$.}
\label{Table:Experiment1:H(Curl)}
\end{table}
\begin{table}
  \begin{center}
  \begin{tabular}{c|r|r r|cc|c}
  level & cores & elements & dof & $||u - u_h||_{L^2(\Omega)}$ & eoc & iter \\ \hline
0 &     16 &     96 &     312 & 4.40957E-1 & - & 16 \\
1 &     16 &    1 536 &    4 032 & 4.58250E-1 & - & 26 \\
2 &     128 &    24 576 &    57 600 & 2.83070E-1 & 0.69 & 27 \\
3 &    16 384 &    393 216 &    869 376 & 1.53587E-1 & 0.88 & 28 \\
4 &    32 768 &   6 291 456 &   13 504 512 & 7.95018E-2 & 0.95 & 28 \\
5 &    32 768 &   100 663 296 &   212 877 312 & 4.02155E-2 & 0.98 & 29 \\
6 &    32 768 &  1 610 612 736 &  3 380 674 560 & 2.01713E-2 & 1.00 & 33 \\ \hline
  \end{tabular}
  \end{center}
\caption{Convergence result and iteration numbers for $H(\Div, \om, \KKK)$.}
\label{Table:Experiment1:H(Div)}
\end{table}
\begin{table}
  \begin{center}
  \begin{tabular}{c|r|r r|cc|c}
  level & cores & elements & dof & $||u - u_h||_{L^2(\Omega)}$ & eoc & iter \\ \hline
0 &     16 &     96 &     288 & 3.74239E-1 & - & 9 \\
1 &     16 &    1 536 &    4 224 & 1.94337E-1 & 0.95 & 18 \\
2 &     128 &    24 576 &    64 512 & 9.76605E-2 & 0.99 & 19 \\
3 &    16 384 &    393 216 &   1 007 616 & 4.88885E-2 & 1.00 & 19 \\
4 &    32 768 &   6 291 456 &   15 925 248 & 2.43047E-2 & 1.01 & 19 \\
5 &    32 768 &   100 663 296 &   253 231 104 & 1.20968E-2 & 1.01 & 20 \\
6 &    32 768 &  1 610 612 736 &  4 039 114 752 & 6.03196E-3 & 1.00 & 24 \\ \hline
  \end{tabular}
  \end{center}
\caption{Convergence result and iteration numbers for $H(\div, \om, \RRR^4)$.}
\label{Table:Experiment1:H(div)}
\end{table}
\end{footnotesize}

For $3$-forms, i.e. $d=\div$ we consider the exact solution
\begin{align*}
 u(x) = [ c_1 s_2 s_3 s_4, s_1 c_2 s_3 s_4, s_1 s_2 c_3 s_4, s_1 s_2 s_3 c_4 ]^\top.
\end{align*}
For these lowest order Raviart-Thomas finite elements in 4D, we again
observe the correct convergence rate of one from Table
\ref{Table:Experiment1:H(div)}.  The iteration numbers for the auxiliary space
preconditioner again exhibit a small growth.

\begin{footnotesize}
\begin{table}
  \begin{center}
  \begin{tabular}{c|ccccccccccccc}
          & \multicolumn{13}{c}{weight $\tau$} \\
	  L & $10^{-6}$ & $10^{-5}$ & $10^{-4}$ & $10^{-3}$ & $10^{-2}$ & $10^{-1}$ & 1 & 10 & $10^{2}$ & $10^{3}$ & $10^{4}$ & $10^{5}$ & $10^{6}$ \\ \hline
    0 & 5 & 5 & 6 & 6 & 7 & 7 & 6 & 5 & 9 & 11 & 12 & 12 & 12 \\
    1 & 10 & 11 & 12 & 12 & 12 & 12 & 11 & 9 & 5 & 10 & 12 & 12 & 12 \\
    2 & 15 & 15 & 15 & 15 & 15 & 15 & 15 & 13 & 8 & 9 & 12 & 12 & 12 \\
    3 & 19 & 19 & 19 & 19 & 19 & 19 & 18 & 17 & 13 & 7 & 10 & 12 & 12 \\
    4 & 20 & 20 & 20 & 20 & 20 & 20 & 20 & 19 & 15 & 9 & 9 & 11 & 12 \\
    5 & 23 & 23 & 23 & 23 & 23 & 23 & 22 & 21 & 16 & 11 & 6 & 8 & 11 \\
    6 & 27 & 27 & 27 & 27 & 27 & 26 & 25 & 23 & 17 & 12 & 8 & 5 & 9 \\ \hline
  \end{tabular}
  \end{center}
\caption{Iteration numbers for different number of weights $\tau$ and refinements for the space $H(\grad, \om, \RRR)$.}
\label{Table:Experiment2:H(grad)}
\end{table}
\begin{table}
  \begin{center}
  \begin{tabular}{c|ccccccccccccc}
          & \multicolumn{13}{c}{weight $\tau$} \\
	  L & $10^{-6}$ & $10^{-5}$ & $10^{-4}$ & $10^{-3}$ & $10^{-2}$ & $10^{-1}$ & 1 & 10 & $10^{2}$ & $10^{3}$ & $10^{4}$ & $10^{5}$ & $10^{6}$ \\ \hline
	  0 & 11 & 11 & 12 & 12 & 11 & 11 & 10 & 11 & 16 & 18 & 19 & 19 & 19 \\
	  1 & 14 & 15 & 15 & 15 & 14 & 15 & 15 & 14 & 19 & 24 & 27 & 27 & 27 \\
	  2 & 21 & 19 & 19 & 18 & 17 & 19 & 18 & 15 & 16 & 21 & 28 & 29 & 29 \\
	  3 & 23 & 21 & 20 & 19 & 21 & 23 & 22 & 20 & 15 & 18 & 22 & 26 & 26 \\
	  4 & 22 & 21 & 20 & 21 & 23 & 24 & 25 & 22 & 16 & 14 & 20 & 25 & 28 \\
	  5 & 22 & 23 & 25 & 24 & 27 & 29 & 28 & 25 & 18 & 14 & 16 & 22 & 29 \\
	  6 & 28 & 29 & 29 & 31 & 32 & 33 & 35 & 29 & 22 & 14 & 13 & 19 & 27 \\ \hline
  \end{tabular}
  \end{center}
\caption{Iteration numbers for different number of weights $\tau$ and refinements for the space $H(\Curl, \om, \RRR^4)$.}
\label{Table:Experiment2:H(curl)}
\end{table}
\begin{table} 
  \begin{center}
  \begin{tabular}{c|ccccccccccccc}
          & \multicolumn{13}{c}{weight $\tau$} \\
	  L & $10^{-6}$ & $10^{-5}$ & $10^{-4}$ & $10^{-3}$ & $10^{-2}$ & $10^{-1}$ & 1 & 10 & $10^{2}$ & $10^{3}$ & $10^{4}$ & $10^{5}$ & $10^{6}$ \\ \hline
	  0 & 12 & 13 & 14 & 14 & 15 & 15 & 16 & 17 & 21 & 23 & 23 & 23 & 23 \\
	  1 & 20 & 21 & 22 & 23 & 24 & 26 & 26 & 26 & 25 & 31 & 34 & 34 & 34 \\
	  2 & 20 & 21 & 23 & 24 & 26 & 27 & 27 & 27 & 26 & 27 & 34 & 35 & 35 \\
	  3 & 20 & 21 & 23 & 24 & 26 & 27 & 28 & 27 & 26 & 24 & 27 & 30 & 31 \\
	  4 & 20 & 21 & 22 & 24 & 26 & 27 & 28 & 27 & 26 & 25 & 24 & 28 & 30 \\
	  5 & 22 & 22 & 24 & 25 & 27 & 28 & 29 & 28 & 25 & 25 & 24 & 25 & 31 \\
	  6 & 28 & 28 & 30 & 30 & 31 & 32 & 33 & 30 & 26 & 23 & 24 & 21 & 28 \\ \hline
  \end{tabular}
  \end{center}
\caption{Iteration numbers for different number of weights $\tau$ and refinements for the space $H(\Div, \om, \KKK)$.}
\label{Table:Experiment2:H(Div)}
\end{table}
\begin{table}
  \begin{center}
  \begin{tabular}{c|ccccccccccccc}
          & \multicolumn{13}{c}{weight $\tau$} \\
	  L & $10^{-6}$ & $10^{-5}$ & $10^{-4}$ & $10^{-3}$ & $10^{-2}$ & $10^{-1}$ & 1 & 10 & $10^{2}$ & $10^{3}$ & $10^{4}$ & $10^{5}$ & $10^{6}$ \\ \hline
	  0 & 8 & 8 & 9 & 9 & 9 & 9 & 9 & 10 & 13 & 14 & 14 & 14 & 14 \\
	  1 & 20 & 20 & 19 & 19 & 18 & 17 & 18 & 18 & 17 & 16 & 17 & 17 & 17 \\
	  2 & 22 & 21 & 20 & 19 & 18 & 18 & 19 & 20 & 20 & 16 & 17 & 17 & 17 \\
	  3 & 21 & 20 & 18 & 17 & 17 & 18 & 19 & 20 & 20 & 18 & 15 & 16 & 16 \\
	  4 & 19 & 18 & 17 & 16 & 17 & 18 & 19 & 20 & 19 & 19 & 16 & 15 & 16 \\
	  5 & 18 & 20 & 20 & 18 & 18 & 19 & 20 & 20 & 19 & 18 & 18 & 14 & 16\\
	  6 & 22 & 22 & 22 & 22 & 23 & 23 & 24 & 23 & 20 & 16 & 17 & 16 & 15\\ \hline
  \end{tabular}
  \end{center}
\caption{Iteration numbers for different number of weights $\tau$ and refinements for the space $H(\div, \om, \RRR^4)$.}
\label{Table:Experiment2:H(div)}
\end{table}  
\end{footnotesize}

\subsection{Parameter robustness}
\label{SubSection:ParameterStudies}

In the following experiments we will study the preconditioners when a
parameter $\tau > 0$ is involved, namely instead
of~\eqref{eq:numerprob}, we consider the following problem:
Find
$u \in H(d, \om, \F k),$ such that
\begin{align*}
  \tau (u,v) + ( du, dv) = \langle F, v \rangle \qquad \text{for all } v \in H(d, \om, \F k),
\end{align*}
where $F$ is set for each $k$ as described previously. All other
parameters, including the domain and stopping criterion, are set as in
the previous experiments.  The results summarized in Tables
\ref{Table:Experiment2:H(grad)}-\ref{Table:Experiment2:H(div)} show 
iteration numbers for the preconditioned conjugate gradient method.
For weights $\tau$ ranging from $10^{-6}$ to $10^6$ we observe quite
small iteration numbers which vary only slightly with $\tau$.  For
small values of~$\tau$, these observations are consistent with the
analytical conclusions of Theorem~\ref{thm:main}.



\section*{Conclusion} 

We presented the auxiliary space preconditioning technique in
arbitrary dimensions. The presentation extends previous results of
Hiptmair and Xu \cite{HiptmXu07} using  recent estimates
on regularized homotopy 
operators~\cite{CostaMcInt09, MitreMitreMonni08} and 
recent developments in finite
element exterior calculus~\cite{ArnolFalkWinth10}.  Although, we only
analyzed the additive version of the auxiliary space preconditioners,
their multiplicative versions can be similarly derived and analyzed.

This work also provides an implementation of the 4D auxiliary space
preconditioners (currently available as one of the public domain
development branches of \cite{mfem}). During this work, we also
implemented the finite element subspaces (currently the lowest
order ones) of the 4D de Rham complex, their canonical interpolants
using proxy identities, and attendant 4D mesh operations. Use of this
technology for spacetime applications and the addition of geometric
multigrid to the tool set are subjects of ongoing research.


\begin{thebibliography}{10}

\bibitem{Andreev13}
{\sc R.~Andreev}, {\em Stability of sparse space-time finite element
  discretizations of linear parabolic evolution equations}, IMA J. Numer.
  Anal., 33 (2013), pp.~242--260.

\bibitem{ArnolFalkWinth06}
{\sc D.~N. Arnold, R.~S. Falk, and R.~Winther}, {\em Finite element exterior
  calculus, homological techniques, and applications}, Acta Numer., 15 (2006),
  pp.~1--155.

\bibitem{ArnolFalkWinth09}
\leavevmode\vrule height 2pt depth -1.6pt width 23pt, {\em Geometric
  decompositions and local bases for spaces of finite element differential
  forms}, Comput. Methods Appl. Mech. Engrg., 198 (2009), pp.~1660--1672.

\bibitem{ArnolFalkWinth10}
\leavevmode\vrule height 2pt depth -1.6pt width 23pt, {\em Finite element
  exterior calculus: from {H}odge theory to numerical stability}, Bull. Amer.
  Math. Soc. (N.S.), 47 (2010), pp.~281--353.

\bibitem{BabuskaJanik89}
{\sc I.~Babu\v{s}ka and T.~Janik}, {\em The h-p version of the finite element
  method for parabolic equations. {P}art {I}. the p-version in time}, Numerical
  Methods for Partial Differential Equations, 5 (1989), pp.~363--399.

\bibitem{BabuskaJanik90}
\leavevmode\vrule height 2pt depth -1.6pt width 23pt, {\em The h-p version of
  the finite element method for parabolic equations. {II}. the h-p version in
  time}, Numerical Methods for Partial Differential Equations, 6 (1990),
  pp.~343--369.

\bibitem{BankVassilevskiZikatanov17}
{\sc R.~E. Bank, P.~S. Vassilevski, and L.~T. Zikatanov}, {\em Arbitrary
  dimension convection--diffusion schemes for space--time discretizations},
  Journal of Computational and Applied Mathematics, 310 (2017), pp.~19--31.

\bibitem{Bey00}
{\sc J.~Bey}, {\em Simplicial grid refinement: on {F}reudenthal's algorithm and
  the optimal number of congruence classes}, Numer. Math., 85 (2000),
  pp.~1--29.

\bibitem{BirmaSolom90}
{\sc M.~Birman and M.~Solomyak}, {\em Construction in a piecewise smooth domain
  of a function of the class {$H^2$} from the value of the conormal
  derivative}, J. Math. Sov., 49 (1990), pp.~1128---1136.

\bibitem{bochev_et_al}
{\sc P.~Bochev, J.~Hu, C.~Siefert, and R.~Tuminaro}, {\em An algebraic
  multigrid approach based on a compatible gauge reformulation of {M}axwell's
  equations}, SIAM J. Sci. Comput., 31 (2008), pp.~557--583.

\bibitem{yunrong}
{\sc P.~Bochev, C.~Siefert, R.~Tuminaro, J.~Xu, and Y.~Zhu}, {\em Compatible
  gauge approaches for {$H({\rm div})$} equations}, in CSRI Summer Proceedings,
  2007.

\bibitem{BrambPasciStein02}
{\sc J.~H. Bramble, J.~E. Pasciak, and O.~Steinbach}, {\em On the stability of
  the {$L^2$} projection in {$H^1(\Omega)$}}, Math. Comp., 71 (2002),
  pp.~147--156.

\bibitem{ChrisWinth08}
{\sc S.~H. Christiansen and R.~Winther}, {\em Smoothed projections in finite
  element exterior calculus}, Math. Comp., 77 (2008), pp.~813--829.

\bibitem{CostaMcInt09}
{\sc M.~Costabel and A.~McIntosh}, {\em On {Bogovski{\u \i}} and regularized
  {Poincar{\'e}} integral operators for {de Rham} complexes on {Lipschitz}
  domains}, Mathematische Zeitschrift,  (2010),
  pp.~DOI:10.1007/s00209--009--0517--8.

\bibitem{DemkoGopalNagar17}
{\sc L.~Demkowicz, J.~Gopalakrishnan, S.~Nagaraj, and P.~Sepulveda}, {\em A
  spacetime {DPG} method for the {Schr\"odinger} equation}, SIAM J. Numer.
  Anal., 55 (2017), pp.~1740--1759.

\bibitem{DemloHiran14}
{\sc A.~Demlow and A.~N. Hirani}, {\em A posteriori error estimates for finite
  element exterior calculus: the de {R}ham complex}, Found. Comput. Math., 14
  (2014), pp.~1337--1371.

\bibitem{FalgoFriedKolev14}
{\sc R.~D. Falgout, S.~Friedhoff, T.~V. Kolev, S.~P. MacLachlan, and J.~B.
  Schroder}, {\em Parallel time integration with multigrid}, SIAM Journal on
  Scientific Computing, 36 (2014), pp.~C635--C661.

\bibitem{Freudenthal42}
{\sc H.~Freudenthal}, {\em Simplizialzerlegungen von beschr\"ankter
  {F}lachheit}, Ann. of Math. (2), 43 (1942), pp.~580--582.

\bibitem{GopalSepul17}
{\sc J.~Gopalakrishnan and P.~Sep\'ulveda}, {\em A spacetime {DPG} method for
  acoustic waves}, Preprint arXiv:1709.08268,  (2017).

\bibitem{Hansbo94}
{\sc P.~Hansbo}, {\em Space-time oriented streamline diffusion methods for
  nonlinear conservation laws in one dimension}, Comm. Numer. Methods Engrg.,
  10 (1994), pp.~203--215.

\bibitem{HensoYang02}
{\sc V.~Henson and U.~Yang}, {\em {BoomerAMG:} a parallel algebraic multigrid
  solver and preconditioner}, Applied Numerical Mathematics, 41 (2002),
  pp.~155--177.

\bibitem{Hiptm99}
{\sc R.~Hiptmair}, {\em Canonical construction of finite elements}, Math.
  Comp., 68 (1999), pp.~1325--1346.

\bibitem{HiptmLiZhou12}
{\sc R.~Hiptmair, J.~Li, and J.~Zhou}, {\em Universal extension for {S}obolev
  spaces of differential forms and applications}, Journal of Functional
  Analysis, 263 (2012), pp.~364--382.

\bibitem{HiptmXu07}
{\sc R.~Hiptmair and J.~Xu}, {\em Nodal auxiliary space preconditioning in
  {${\bf H}({\bf curl})$} and {${\bf H}({\rm div})$} spaces}, SIAM J. Numer.
  Anal., 45 (2007), pp.~2483--2509.

\bibitem{hypre}
{\em \emph{hypre}: High performance preconditioners}.
\newblock http://www.llnl.gov/CASC/hypre/.

\bibitem{Jense04}
{\sc M.~Jensen}, {\em Discontinuous {Galerkin} Methods for {Friedrichs} Systems
  with {Irregular} Solutions}, PhD thesis, University of Oxford, 2004.

\bibitem{JohnsonNaevertPitkaeranta84}
{\sc C.~Johnson, U.~N{\"a}vert, and J.~Pitk{\"a}ranta}, {\em Finite element
  methods for linear hyperbolic problems}, Comput. Methods Appl. Mech. Engrg.,
  45 (1984), pp.~285--312.

\bibitem{JohnsonSaranen86}
{\sc C.~Johnson and J.~Saranen}, {\em Streamline diffusion methods for the
  incompressible {E}uler and {N}avier-{S}tokes equations}, Math. Comp., 47
  (1986), pp.~1--18.

\bibitem{KolevVassi09}
{\sc T.~V. Kolev and P.~S. Vassilevski}, {\em Parallel auxiliary space {AMG}
  for {$H({\rm curl})$} problems}, J. Comput. Math., 27 (2009), pp.~604--623.

\bibitem{KolevVassi12}
\leavevmode\vrule height 2pt depth -1.6pt width 23pt, {\em Parallel auxiliary
  space {AMG} solver for {$H({\rm div})$} problems}, SIAM J. Sci. Comput., 34
  (2012), pp.~A3079--A3098.

\bibitem{LangerMooreNeumueller16}
{\sc U.~Langer, S.~Moore, and M.~Neum\"uller}, {\em Space-time isogeometric
  analysis of parabolic evolution equations}, Comput. Methods Appl. Mech.
  Engrg., 306 (2016), pp.~342--363.

\bibitem{LarssonMolten17}
{\sc S.~Larsson and M.~Molten}, {\em Numerical solution of parabolic problems
  based on a weak space-time formulation}, Comput. Methods Appl. Math., 17
  (2017), pp.~65--84.

\bibitem{mfem}
{\em {MFEM}: {M}odular finite element methods library}.
\newblock {{\tt \small http://mfem.org}}.

\bibitem{MitreMitreMonni08}
{\sc D.~Mitrea, M.~Mitrea, and S.~Monniaux}, {\em The {P}oisson problem for the
  exterior derivative operator with {D}irichlet boundary condition in nonsmooth
  domains}, Commun. Pure Appl. Anal., 7 (2008), pp.~1295--1333.

\bibitem{Mollet14}
{\sc C.~Mollet}, {\em Stability of {P}etrov-{G}alerkin discretizations:
  Application to the space-time weak formulation for parabolic evolution
  problems}, Computational Methods in Applied Mathematics, 14 (2014),
  pp.~231--255.

\bibitem{Nepom07}
{\sc S.~Nepomnyaschikh}, {\em Domain decomposition methods}, in Lectures on
  advanced computational methods in mechanics, vol.~1 of Radon Ser. Comput.
  Appl. Math., Walter de Gruyter, Berlin, 2007, pp.~89--159.

\bibitem{NeumuellerSteinbach11}
{\sc M.~Neum\"uller and O.~Steinbach}, {\em Refinement of flexible space--time
  finite element meshes and discontinuous {G}alerkin methods}, Comput. Visual.
  Sci., 14 (2011), pp.~189--205.

\bibitem{NeumuellerVassilevskiVilla2017}
{\sc M.~Neum{\"u}ller, P.~S. Vassilevski, and U.~E. Villa}, {\em Space-time
  {C}{F}{O}{S}{L}{S} methods with {A}{M}{G}e upscaling}, in Domain
  Decomposition Methods in Science and Engineering XXIII, C.-O. Lee, X.-C. Cai,
  D.~E. Keyes, H.~H. Kim, A.~Klawonn, E.-J. Park, and O.~B. Widlund, eds.,
  Cham, 2017, Springer International Publishing, pp.~253--260.

\bibitem{zhao_pasciak}
{\sc J.~E. Pasciak and J.~Zhao}, {\em {Overlapping Schwarz Methods in $H(curl)$
  on Nonconvex Domains}}, East West J. Num. Anal., 10 (2002), pp.~221--234.

\bibitem{Schob08b}
{\sc J.~Sch{\"o}berl}, {\em A posteriori error estimates for {M}axwell
  equations}, Math. Comp., 77 (2008), pp.~633--649.

\bibitem{SchwabStevenson09}
{\sc C.~Schwab and R.~Stevenson}, {\em Space--time adaptive wavelet methods for
  parabolic evolution problems}, Math. Comput., 78 (2009), pp.~1293--1318.

\bibitem{Smears16}
{\sc I.~Smears}, {\em Robust and efficient preconditioners for the
  discontinuous galerkin time-stepping method}, IMA J. Numer. Anal., 00 (2016),
  pp.~1--25.

\bibitem{Steinbach15}
{\sc O.~Steinbach}, {\em Space-time finite element methods for parabolic
  problems}, Comput. Meth. Appl. Math., 15 (2015), pp.~551--566.

\bibitem{UrbanPatera14}
{\sc K.~Urban and A.~T. Patera}, {\em An improved error bound for reduced basis
  approximation of linear parabolic problems}, Math. Comput., 83 (2014),
  pp.~1599--1615.

\bibitem{VegtVen02}
{\sc J.~J.~W. van~der Vegt and H.~van~der Ven}, {\em Space-time discontinuous
  {G}alerkin finite element method with dynamic grid motion for inviscid
  compressible flows. {I}. {G}eneral formulation}, J. Comput. Phys., 182
  (2002), pp.~546--585.

\bibitem{Whitn57}
{\sc H.~Whitney}, {\em Geometric integration theory}, Princeton University
  Press, Princeton, N. J., 1957.

\bibitem{Xu96}
{\sc J.~Xu}, {\em The auxiliary space method and optimal multigrid
  preconditioning techniques for unstructured grids}, Computing, 56 (1996),
  pp.~215--235.

\end{thebibliography}
\end{document}